\documentclass[12pt]{article}

\usepackage{amsmath, amssymb, amsthm, mathrsfs,physics}
\usepackage{hyperref}  
\usepackage{enumitem}  
\usepackage{geometry}  
\hypersetup{colorlinks=true,linkcolor=red, anchorcolor=blue, citecolor=blue, urlcolor=red, filecolor=magenta, pdftoolbar=true}

\numberwithin{equation}{section}

\newtheorem{theorem}{Theorem}[section]
\newtheorem{lemma}[theorem]{Lemma}
\newtheorem{proposition}[theorem]{Proposition}
\newtheorem{corollary}[theorem]{Corollary}

\theoremstyle{definition}
\newtheorem{definition}[theorem]{Definition}

\theoremstyle{remark}
\newtheorem{remark}[theorem]{Remark}
\newtheorem{example}[theorem]{Example}

\newcommand{\N}{\mathbb{N}}  
\newcommand{\Z}{\mathbb{Z}}  

\allowdisplaybreaks

\title{Stable Initial-State Recovery from Dynamical Samples using Nagy-Type and Pollard-Hilding-Type  Frame Perturbations}
\author{Ruchi and Lalit Kumar Vashisht$^*$\\
	\small{Department of Mathematics, University of Delhi}\\
	\small{Emails: rgarg@maths.du.ac.in \quad $\&$ \quad lalitkvashisht@gmail.com}}
\date{\small{$^*$Corresponding author} \\October 26, 2026}

\begin{document}
	
	\maketitle
	
		\begin{abstract}
		Stimulated by Aldroubi and his collaborator's recent work on dynamical sampling, we consider a homogeneous discrete dynamical system of the form
		\begin{align*}
			f_n = A f_{n-1}= A^{n}f, \quad f_0 = f,
		\end{align*}
		where $A$ is a bounded linear operator on a separable Hilbert space $\mathcal{H}$, which is known as the evolution operator, and $f_0 \in \mathcal{H}$ is the unknown  initial state. The associated dynamical samples are given by the collection $\{\langle A^n f, g \rangle  : g \in \mathcal{G}, 0 \leq  n < L(g)\}$, where $\mathcal{G} \subset \mathcal{H}$, is a finite or countable sampling set, and $L : \mathcal{G} \rightarrow \N \cup \{\infty\}$ is a function. We analyze the stability of perturbed dynamical sampling systems in the sense of Nagy and Pollard–Hilding. Sufficient conditions for the stable recovery of an initial state from perturbed dynamical samples obtained by changing the sampling vector, the evolution operator, or simultaneously both within the framework of Nagy-type and Pollard–Hilding-type  perturbation of frames are established. In our work, the sampling set $\mathcal{G}$ consists of a single vector $g$, that is, $|\mathcal{G}|=1$, and $L(g)=\infty$. \\

	\end{abstract}

	\section{Introduction}
	\label{sec:intro}
	Dynamical sampling is a growing area of research that is centered at the confluence of sampling theory, frame theory, operator theory, and signal processing. The field of dynamical sampling has developed rapidly since the pioneering work of Lu et al. \cite{SRDL} and Aldroubi et al. \cite{DSIS, DSTS}, and now includes several active directions such as initial-state recovery, system identification, and source recovery \cite{DSIS,DSTS,ERSS,KSMDS,RRDSE, CFSRDS,PSDDDS, Ruchi, SIDS}. The central objective of dynamical sampling is to recover a function or signal from samples generated through the action of an evolution operator on a collection of sampling vectors.  More precisely, let $A$ be a bounded linear operator on a separable Hilbert space $\mathcal{H}$ and $\mathcal{W}$ be a closed subspace of $\mathcal{H}$. A general discrete dynamical system is of the form
	\begin{align*}
		f_n = A f_{n-1}+w, \, \,  f_0 = f, \,  w \in \mathcal{W},
	\end{align*}
	where $A$ is called the evolution operator, $f_0 \in \mathcal{H}$ is the initial state, and $w \in \mathcal{W}$ is the source term or forcing term. Given a finite or countable subset $\mathcal{G} \subset \mathcal{H}$, known as the sampling set, and a function $L : \mathcal{G} \rightarrow \N \cup \{\infty\}$, the associated measurements are $	\{\langle f_n, g \rangle  : g \in \mathcal{G}, 0 \leq  n < L(g)\}$.
	These measurements are referred to as space-time samples or dynamical samples. The general dynamical sampling problem is to determine conditions on $A$, $\mathcal{G}$ and $L : \mathcal{G} \rightarrow \N \cup \{\infty\}$, that allow the recovery of the unknown components of the system from the given measurements. Depending on the setting, the unknown quantity may be the initial state $f$, the evolution operator $A$, or the source term $w$. In applications, it is important that the sampling data enable stable recovery that remains robust in the presence of perturbations and noise. By stable recovery, we mean that there exists a well-defined bounded linear reconstruction operator $\mathscr{R}: \ell^{2}(S) \rightarrow \mathcal{H}$ satisfying $\mathscr{R}(\{\langle f_n, g \rangle \}_{(g,n) \in S})=f (\text{or} \,w), \, \text{for all} \, f (\text{or} \,w) \in \mathcal{H}$, where $S=\{(g,k) : g \in \mathcal{G}, 0 \leq  k < L(g)\}$ and under suitable conditions on $A$, $L$, and $\mathcal{G}$, the collection of measurements $\{\langle f_n, g \rangle  : (g,n) \in S\}$ forms a square-summable sequence in $\ell^{2}(S)$. An important special case arises when the source term vanishes, i.e., $w=0$. In this homogeneous setting, the system reduces to  $f_n=A f_{n-1}=A^{n}f$, and the associated dynamical samples become $	\{\langle A^n f, g \rangle  : g \in \mathcal{G}, 0 \leq  n < L(g)\}$.
	The corresponding recovery problem is to determine conditions on $A$, $\mathcal{G}$ and $L $ that allow the reconstruction of the initial state $f \in \mathcal{H}$ from these measurements. In this case, the problem is closely connected with frames (see below), since $\langle A^n f, g \rangle = \langle  f, (A^{*})^{n}g \rangle $, where $A^{*}$ denotes the Hilbert-adjoint of $A$. Consequently, recovery of the initial state is equivalent to determining whether the collection $\{(A^{*})^{n}g\}_{g \in \mathcal{G}, \;  0 \leq n < L(g)}$, is complete or forms a frame for $\mathcal{H}$. Such systems arise naturally in applications related to signal processing, wireless sensor networks, and time-varying signals \cite{SRDL}. Completeness guarantees uniqueness of recovery, while the frame property ensures stable reconstruction. This connection was  established due to Aldroubi and Petrosyan in the following proposition:
	\begin{proposition}$\cite{DSSIO}$\label{P2_P2.3}
		In a separable Hilbert space $\mathcal{H}$, the following statements hold:
		\begin{enumerate}
			\item  Any $f \in \mathcal{H}$ can be recovered from $\{\langle A^{n}f,g\rangle\}_{g \in \mathscr{G}, \;  0 \leq n < L(g)}$ if and
			only if the system $\{(A^{*})^{n}g\}_{g \in \mathscr{G}, \;  0 \leq n < L(g)}$ is complete in $\mathcal{H}$.
			\item Any $f \in \mathcal{H}$ can be recovered from $\{\langle A^{n}f,g\rangle\}_{g \in \mathscr{G}, \; 0 \leq n < L(g)}$ in a stable way if and
			only if the system $\{(A^{*})^{n}g\}_{g \in \mathscr{G}, \; 0 \leq n < L(g)}$  is a frame in $\mathcal{H}$.
		\end{enumerate}
	\end{proposition}
	
	Recall that a  discrete \emph{frame} for separable Hilbert space $\mathcal{H}$ is a  countable collection  $\{\varphi_k\}_{k \in \mathbb{I}}$ of members of  $\mathcal{H}$ such that the following inequality
	(called the \emph{frame inequality}) holds for some positive real numbers $A_o$ and $B_o$:
	\begin{align}\label{b01.1}
		A_o \|\varphi\|^2\leq  \sum_{k \in \mathbb{I}}|\langle \varphi, \varphi_k\rangle|^2 \leq B_o \|\varphi\|^2 \ \text{for all} \ \varphi \in \mathcal{H}.
	\end{align}
	$A_o$ and $B_o$ are called the \emph{lower frame bound} and \emph{upper frame bound} of the frame $\{\varphi_k\}_{k \in \mathbb{I}}$,  respectively. We say that  $\{\varphi_k\}_{k \in \mathbb{I}}$ satisfies the
	\emph{lower frame condition} if the left inequality in  \eqref{b01.1} holds. If $\{\varphi_k\}_{k \in \mathbb{I}}$ satisfies  only upper inequality in \eqref{b01.1}, then we say that $\{\phi_k\}_{k \in \mathbb{I} }$   is  a \emph{Bessel sequence} with the \emph{Bessel bound} $B_o$, or it satisfies the \emph{upper frame condition}.
	If $\{\varphi_k\}_{k \in \mathbb{I}}$ is a Bessel sequence, then the map $T :\ell^2(\mathbb{I}) \rightarrow \mathcal{H}$, given by
	\begin{align*}
		T\{c_{k}\}_{k \in \mathbb{I}}=\sum_{k \in \mathbb{I}}c_k \varphi_k
	\end{align*}
	is called the \emph{synthesis operator} or \emph{pre-frame operator} of $\{\varphi_k\}_{k \in \mathbb{I}}$. Its Hilbert-adjoint operator $T^*: \mathcal{H} \rightarrow  \ell^2(\mathbb{I})$  is called the \emph{analysis operator} of $\{\varphi_k\}_{k \in \mathbb{I}}$ that is given by $T^*\varphi = \{\langle \varphi, \varphi_k\rangle\}_{k \in \mathbb{I}}$. The \emph{frame operator} of the Bessel sequence $\{\varphi_k\}_{k \in \mathbb{I}}$ is the composition:
	\begin{align*}
		S = TT^*: \mathcal{H} \rightarrow \mathcal{H}, \quad S \varphi = \sum_{k \in \mathbb{I}}\langle \varphi, \varphi_k\rangle \varphi_k.
	\end{align*}
	The frame operator of a Bessel sequence is bounded, linear, and self-adjoint; but may not be  invertible on $\mathcal{H}$. If $\{\varphi_k\}_{k \in \mathbb{I}}$ is a frame for $\mathcal{H}$, then the frame operator is invertible on  $\mathcal{H}$. This gives the stable reconstruction formula as a series, not necessarily unique, for each element of the space $\mathcal{H}$:
	\begin{align*}
		\varphi = SS^{-1}\varphi =\sum_{k \in \mathbb{I}} \langle \varphi, S^{-1}\varphi_k \rangle \varphi_k  \ \text{for every} \  \varphi \in \mathcal{H}.
	\end{align*}
	Scalars $\langle \varphi, S^{-1}\varphi_k \rangle $, $k \in \mathbb{I}$, are called the \emph{frame coefficients}. Frame coefficients associated with a frame contain all information of signals in the underlying signal space. For basic theory of frames and its applications, we refer to texts  by Christensen \cite{Ole Christ.},  Han \cite{BHan2}, Heil \cite{C.Heil},  Krivoshein, Protasov   Skopina \cite{KPSA},  Young \cite{Young}, Zhang and Jorgensen  \cite{Zpelle}.
	
	One of the natural questions in dynamical sampling is the stability of reconstruction procedures under perturbations. In practical situations, the underlying evolution operator governing the dynamics, the sampling vector, or both simultaneously may be affected by modeling errors, noise, or numerical inaccuracies. Consequently, it becomes important to determine whether the stable recovery properties persist when the underlying dynamical sampling system is perturbed. This motivates the study of perturbation theory in the context of dynamical sampling systems in Hilbert spaces. In particular, we analyze the stable reconstruction of the initial state of a homogeneous discrete dynamical system from perturbed dynamical samples obtained through perturbations of the sampling vector, the evolution operator, or both simultaneously. The analysis is carried out using Nagy-type and Pollard–Hilding-type perturbation results for frames. Throughout this work, we restrict ourselves to the case where $\mathcal{G}=\{g\}$ and $L(g)=\infty$.
	
	\subsection{Related work}
	The field of dynamical sampling is structured around three core pillars: initial state recovery or space-time trade-off \cite{DSIS, DSTS, ERSS}, system identification \cite{KSMDS, SIDS}, and source recovery \cite{RRDSE, CFSRDS, PSDDDS, Ruchi}. The conceptual roots of dynamical sampling can be traced back to the work of Lu et al. on spatiotemporal sampling of diffusion fields, where the trade-off between spatial and temporal measurements was investigated in the context of signal processing and sensor networks \cite{SRDL}. A systematic mathematical formulation of dynamical sampling was introduced by Aldroubi et al. in 2013, where they studied the recovery of signals evolving under an operator from coarse spatial samples taken over time \cite{DSIS, DSTS}.

	The fundamental problem of initial-state recovery was studied extensively in \cite{DSTS}, where the homogeneous setting $w=0$ was considered. In this work, the sampling vectors were chosen from finite or countable orthonormal bases of the underlying Hilbert space, and the evolution operator was taken to be a discrete convolution operator. The authors established conditions under which a signal can be recovered from its space-time samples and highlighted the role of the trade-off between spatial and temporal sampling densities. In \cite{Ald00}, the authors studied the problem in both finite and infinite-dimensional settings. For  the finite-dimensional case, they obtained explicit reconstruction conditions for diagonalizable as well as general matrices acting on $\mathbb{C}^d$, and showed that vectors can be recovered from suitable space-time samples by appropriately choosing spatial sampling locations and temporal measurements. In the infinite-dimensional setting, the analysis was carried out for certain classes of self-adjoint operators on $\ell^2(\mathbb{N})$. Another important development in initial-state recovery involved operator orbits and iterative systems of the form $\{A^{n}g\}_{g \in \mathcal{G}, \;  n \geq 0}$, where  researchers investigated and analyzed their frame and Bessel properties, thereby linking dynamical sampling with frame theory and operator orbits \cite{DSSIO, DSFIS}. Since then, the connection between dynamical sampling and frames generated by iterative actions of operators has become an active area of research. The field subsequently expanded to system identification, where the evolution operator itself is unknown. Researchers also used generalized Krylov subspace methods to recover the unknown evolution operator and its spectrum from partial space-time observations \cite{KSMDS}, and studied dynamical sampling systems generated by iterative actions of normal operators \cite{ICNO}. In 2017, Sui Tang studied system identification in dynamical sampling and introduced reconstruction methods connected with convolution operators and Prony-type techniques \cite{SIDS}. Applications of dynamical sampling to bandlimited functions can be found in \cite{SFBF, AUla}.

	Another significant direction in the field is source recovery, where the objective is to identify unknown forcing terms driving the dynamics. More recent works have studied dynamical sampling in the presence of additive noise \cite{DSARN}, established recovery conditions for periodic and spatially constant sources in the discrete setting \cite{DSRSCS}, and developed predictive algorithms in the continuous setting for detecting burst-like forcing terms even in the presence of noise and measurement errors \cite{PADBF}. These approaches are robust under Lipschitz-continuous background noise and measurement perturbations. In \cite{CDSR}, the authors studied dynamical sampling on $\ell^2(\mathbb{\Z})$, where the sampling vectors are chosen from the standard basis and the evolution operator is a convolution operator generated by a kernel in $\ell^1(\mathbb{\Z})$.
	By 2025, the field reached another milestone with the appearance of a comprehensive survey on dynamical sampling \cite{DSS}.  Beyond these three principal directions, dynamical sampling is increasingly finding applications in graph signal processing and mobile sampling, where sensors move along prescribed trajectories \cite{GongL, HuangII, Yao}. For applications of dynamical sampling in tensor products, we refer to Zhang and Jorgensen \cite{WangY}.

	In some situations, such as when there is a loss of frame coefficients during signal processing, we need to process the signal using a different family of vectors that can preserve both frame conditions. Perturbation theory, which is concerned with a family of vectors in Banach spaces that are close to a given family of vectors with some fundamental properties, leaving the family with the same property, has been one of the most active and important topics in pure mathematics over the past seven decades. The earliest research work in papers  \cite{NCTPWT, Pollard, BSPWC} on Paley-Wiener-type, Nagy-type and Pollard–Hilding-type showed that completeness and basic sequences are preserved under these variants of perturbation. Invariance of fundamental properties under different types of perturbation is highly applicable in differential equations and operator theory \cite{Kato}. In a normed space $\mathcal{X}$, the stability of completeness of sequences  in the sense of Paley and Wiener says that  ``if a sequence $\{y_k\}_{k=1}^{\infty} \subset \mathcal{X}$ is ``near'' a complete sequence  $\{x_k\}_{k=1}^{\infty} \subset \mathcal{X}$ in some suitable sense, then $\{y_k\}_{k=1}^{\infty}$ is also complete in $\mathcal{X}$''. More precisely, we say that two sequences $\{x_k\}_{k=1}^{\infty}$ and $\{y_k\}_{k=1}^{\infty}$ in a normed space  $\mathcal{X}$ satisfies the \emph{Paley-Wiener property}, or closed in the sense of Paley-Wiener, if there exists a real number $\lambda \in (0, 1)$ such that $\|\sum_{j=1}^{n} a_j(x_j- y_j)\| \leq \lambda \|\sum_{j=1}^{n} a_j x_j\|$ holds for any scalars $a_1, a_2, \cdots a_n$.
	For Paley-Wiener-type stability for completeness of sequences, we refer to \cite{NCTPWT, Pollard}.
	In \cite{BSPWC},  Retherford generalized stability results in the sense of Paley-Wiener for basic sequences and completeness of sequences in complete linear metric spaces and normed spaces that are defined as follows:
	\begin{definition}\cite{BSPWC} \label{concls009} Two sequences $\{x_k\}_{k=1}^{\infty}$ and $\{y_k\}_{k=1}^{\infty}$ in a normed space $\mathcal{X}$ are said to have the property
		\begin{enumerate}
			
			\item 	\label{PH_Def} for \textbf{ Pollard-Hilding} if for each positive real number $k$, there exist real numbers $\lambda_1$, $\lambda_2$ with $0 \leq \lambda_i < \text{min}\{1, 2^{1- \frac{1}{k}}\}, i= 1, 2$, such that
			\begin{align*}
				\Big\|\sum_{j=1}^{n} a_j(x_j- y_j)\Big\| \leq  \left(\lambda_1 \Big\|\sum_{j=1}^{n} a_j x_j\Big\|^k +  \lambda_2 \Big\|\sum_{j=1}^{n} a_j y_j\Big\|^k\right)^{\frac{1}{k}}
			\end{align*}
			holds for any scalars $a_1, a_2, \cdots a_n$.
			
			\item \label{Nagy_Def}for \textbf{Nagy} if there exist real numbers $\lambda \in [0, 1)$,  $\nu \in [0, 1)$ and $\mu$ with $0 \leq \mu, \,  \mu^2 \leq (1- \lambda)(1-\nu)$ such that
			\begin{align*}
				\Big\|\sum_{j=1}^{n} a_j(x_j- y_j)\Big\|^2 \leq  \lambda \Big\|\sum_{j=1}^{n} a_j x_j\Big\|^2 +  \mu   \Big\|\sum_{j=1}^{n} a_j x_j\Big\| \cdot  \Big\|\sum_{j=1}^{n} a_j y_j\Big\|   +    \nu \Big\|\sum_{j=1}^{n} a_j y_j\Big\|^2
			\end{align*}
			holds for any scalars $a_1, a_2, \cdots a_n$.
		\end{enumerate}
	\end{definition}
	Stability in the sense of Nagy, with $\mu =0$, for completeness of sequences in separable Hilbert spaces was proved by Pollard in \cite{Pollard}. This is generalized Paley-Wiener-type stability for complete sequences in separable Hilbert spaces. Also see a published paper by Hilding \cite{NCTPWT} for a variant of Paley-Wiener-type stability for complete sequences in separable Hilbert spaces.
	
	Recall that frames for a signal space enable the stable reconstruction of each signal within that space. Therefore, its analysis is lacking if the  perturbation of  frames is not studied. Further, in the perturbation of frames, it is important that their fundamental properties  must be invariant.   In \cite{Casz}, Casazza and Kalton extended the Paley-Wiener perturbation theory to Banach spaces. Casazza and Christensen in \cite{POATH} presented some applications of Paley-Wiener stability to frames in separable Hilbert spaces. Favier and   Zalik, in \cite{FZalik}, studied perturbation of frames and Riesz bases in separable Hilbert spaces. The authors of \cite{JindIII} studied  perturbation of frames associated with the Weyl-Heisenberg group and the extended affine group. Some variants of  perturbation of Hilbert-Schmidt frames in terms of weaving vectors  can be found in \cite{JyotiI}. Recently, by using the Fourier transforms technique,  Paley-Wiener-type perturbation of non-uniform frames was studied by authors of  \cite{Haribus}. In locally convex spaces, frame perturbation was studied by authors of \cite{Lalit}. As already observed,  perturbation of frames has potential applications in time-frequency analysis \cite{KG},   mathematical physics \cite{GosoI, JyotiII}, signal processing \cite{BHan2, C.Heil}, distributed signal processing \cite{DeepI, deep2}, etc.
	Our paper contributes   perturbation problems for dynamical sampling systems in separable Hilbert spaces for the stable reconstruction of signals. More precisely, we study the stable recovery of an initial state of a homogeneous discrete dynamical system from perturbed dynamical samples obtained by changing the sampling vector, the evolution operator, or simultaneously both using Nagy-type and Pollard–Hilding-type frame perturbations. In the present paper, the sampling set $\mathcal{G}$ consists of a single vector $g$, that is, $|\mathcal{G}|=1$, and $L(g)=\infty$.
	
		\subsection{Outline of the work}
	
	The paper is outlined as follows:  In Section \ref{sec:prelim P2}, we recall some basic results related to frames and inequalities that will be used in the sequel. Section \ref{P2_sec3} is devoted to the stable recovery of the initial state  of homogeneous dynamical sampling system discussed above under Nagy-type frame perturbation. Theorem \ref{P2_Th3.9'} establishes stable recovery under simultaneous perturbations of the sampling vector and the evolution operator via a unified condition, with single-component perturbations recorded as special cases in Corollary \ref{cor_Nagy}, while Theorem \ref{P2_Th3.9} provides a complementary split-sum formulation of simultaneous perturbation. In Section \ref{P2_sec4}, we investigate the stable recovery of the initial state of the homogeneous dynamical sampling system within the framework of Pollard–Hilding-type frame perturbation. Theorem \ref{P2_Th3.17'} establishes the preservation of stable recovery under simultaneous perturbations of the sampling vector and the evolution operator, while Corollary \ref{cor_PH}  records the corresponding single-perturbation special cases. Theorem \ref{P2_Th3.17} studies simultaneous perturbations via a split-sum framework.

	\section{Preliminaries}\label{sec:prelim P2}
	Throughout the paper,  the symbol $\mathbb{N}$ denotes the set of natural numbers and  $\mathbb{Z}^{+} = \mathbb{N} \bigcup \{0\}$, denotes the set of non-negative integers. $\mathcal{B}(\mathcal{H})$ denotes the space of all bounded linear operators on  $\mathcal{H}$.	In this section, we give some basic results that will be used in the paper. We begin with a characterization of a Bessel sequence in separable Hilbert spaces, which can be found in texts \cite{Ole Christ., C.Heil}.
	\begin{theorem} $\cite[\, p. 189]{C.Heil}$ \label{P2_Th2.4}
		A  sequence $\{f_k\}_{k \in \mathbb{I}}$ of elements in a separable infinite-dimensional complex Hilbert space $\mathcal{H}$ is a Bessel sequence with Bessel bound $B_o$ if and only if the map
		$T :\ell^2(\mathbb{I}) \rightarrow \mathcal{H}$, given by
		\begin{align*}
			T\{c_{k}\}_{k \in \mathbb{I}}=\sum_{k \in \mathbb{I}}c_k f_k
		\end{align*}
		is a well-defined bounded linear operator with $\norm{T} \leq \sqrt{B_o} $.
	\end{theorem}
	The following result says that for a sequence $\{f_k\}_{k \in \mathbb{I}} \subset \mathcal{H}$ to be a Bessel sequence it is enough to check an estimate in terms of the series given in Theorem \ref{P2_Th2.4} for finite sequences.
	\begin{theorem} $\cite[\, p. 191]{C.Heil}$ \label{P2_Th2.5}
		Let $\{f_k\}_{k \in \mathbb{I}}$ be a sequence of elements in $\mathcal{H}$ such that there exists a constant $B_o > 0$ satisfying
		\begin{align*}
			\norm{\sum c_k f_k}^{2} \leq B_o \sum \abs{c_k}^2
		\end{align*}
		for all finite scalar sequences $\{c_k\}_{k \in \mathbb{I}}$. Then the series $ \sum_{k \in \mathbb{I}} c_k f_k $ converges for every sequences $\{c_k\}_{k \in \mathbb{I}} \in \ell^{2}(\mathbb{I})$. Moreover, the sequence $\{f_k\}_{k \in \mathbb{I}}$ is a Bessel sequence with Bessel bound $B_o$.
	\end{theorem}
	
	Let $\mathcal{H}_1$ and $\mathcal{H}_2$ be separable Hilbert spaces and let $\Theta: \mathcal{H}_1 \rightarrow \mathcal{H}_2$ be a bounded linear operator with closed range Ran$\Theta$. Then there exists a bounded linear operator $\Theta^{\dagger}: \mathcal{H}_2 \rightarrow \mathcal{H}_1$ for which $\Theta \Theta^{\dagger}(g) = g$ for all $g \in \text{Ran}\Theta$. The operator $\Theta^{\dagger}$ is called the \emph{pseudo-inverse} of $\Theta$, see \cite[Sec. 2.5]{Ole Christ.}. The following result gives the pseudo-inverse of the synthesis operator.
	\begin{theorem} $\cite[\, p. 131]{Ole Christ.}$ \label{P2_Th2.6}
		Let $\{f_k\}_{k \in \mathbb{I}}$ be a frame for $\mathcal{H}$ with synthesis operator $	T $ and  frame operator $ S $. Then, for every $f \in \mathcal{H}$, we have
		\begin{align*}
			T^{\dagger} f = \{ \langle f, S^{-1} f_k \rangle \}_{k \in \mathbb{I}}.
		\end{align*}
	\end{theorem}
	The following inequalities given in the next three results will be used in the sequel.
	\begin{lemma}$\cite{POATH}$\label{P2_L2.5}
		Let $X$ be a Banach space and  $U : X \rightarrow X$ be a linear operator. Suppose that there exist constants $\lambda_{1}$, $\lambda_{2} \in [0,1)$ such that
		\begin{align*}
			\norm{Ux-x} \leq \lambda_{1} \norm{x} + \lambda_{2} \norm{Ux} \, \,\,  \text{for all} \, \,\,  x \in X.
		\end{align*}
		Then $U$ is bounded and invertible. Moreover,
		\begin{align*}
			\frac{1-\lambda_{1}}{1+ \lambda_{2}} \norm{x} \leq \norm{Ux} \leq \frac{1+\lambda_{1}}{1- \lambda_{2}} \norm{x} \quad \text{and} \quad  \frac{1-\lambda_{2}}{1+ \lambda_{1}} \norm{x} \leq \norm{U^{-1}x} \leq \frac{1+\lambda_{2}}{1- \lambda_{1}} \norm{x} \,\,\,  \text{for all} \, \, x \in X.
		\end{align*}
	\end{lemma}
	
	\begin{theorem}$\cite[p. 204 ]{Math Ineq.}$ \label{Young}
		$[$\textbf{Young's Inequality}$]$
		Let $1<p, q < \infty$ be such that  $\frac{1}{p} + \frac{1}{q}=1$. Then, for all $a$, $b >  0$, we have $ab \leq \frac{a^p}{p} + \frac{b^q}{q} $. As a special case for $p=q=2$, if we replace $a$ by $\left( \sqrt{\epsilon} a \right)$ and $b$ by $\left(\frac{b}{\sqrt{\epsilon}}\right)$ in above inequality, then we get the following $\epsilon$-version of Young’s inequality
		\begin{align*}
			ab \leq \frac{\epsilon}{2}a^{2} + \frac{1}{2 \epsilon} b^{2}, \quad \text{where} \quad \epsilon >0.
		\end{align*}
	\end{theorem}
	
	\begin{theorem}$\cite{NCTPWT}$\label{Theorem 2.8}
		For any $x \geq 0$, $y \geq 0$, we have $x^{k}+ y^{k} \leq c (x+y)^{k}$, where
		\begin{align*}
			c=\begin{cases}
				1, & \;\text{if} \;  k \geq 1,\\
				2^{1-k}, & \;  \text{if} \; k \leq 1 .
			\end{cases}
		\end{align*}
	\end{theorem}

	\section{Dynamical Samples using Nagy-Type  Perturbation of Frames} \label{P2_sec3}
	
		In this section, we apply the Nagy-type perturbation framework to homogeneous dynamical sampling systems to study the stable recovery of the initial state . We begin by establishing a general frame perturbation result in separable Hilbert spaces, which provides explicit frame bounds under quadratic Nagy-type conditions.
	\begin{lemma} \label{L_3.1}
		Let $\{f_n\}_{n=0}^{\infty}$ be a frame for $\mathcal{H}$ with lower and upper frame bounds $\alpha$ and $\beta$, respectively, and let $\{g_n\}_{n=0}^{\infty}$ be a sequence in $\mathcal{H}$. Suppose that for all $l \in \Z^{+}$  and all finite scalar sequences $\{a_n\}_{n=0}^{\infty}$, there exist constants $\lambda$, $\mu$, $\nu  \geq 0$ satisfying
		\begin{align}
			\norm{\sum_{n=0}^{l}a_{n}(f_{n}-g_{n})}^{2} & \leq  \lambda \norm{\sum_{n=0}^{l}a_{n}f_{n}}^{2} + \mu \norm{\sum_{n=0}^{l}a_{n}f_{n}} \, \norm{\sum_{n=0}^{l}a_{n}g_{n}} + \nu \norm{\sum_{n=0}^{l}a_{n}g_{n}}^{2},
			\label{P2_L3_1}
		\end{align}
		and  $\max{\{\lambda+ \frac{\mu \epsilon}{2}, \frac{\mu}{2 \epsilon}+ \nu\}}<1$, where $\epsilon >0$. Then, $\{g_{n}\}_{n=0}^{\infty}$  is a frame for $\mathcal{H}$ with  frame bounds $\alpha\Big(\frac{1-\sqrt{\lambda +\frac{\mu \epsilon}{2}}}{1+ \sqrt{\frac{\mu}{2 \epsilon}+\nu}}\Big)^{2}$ and $ \beta \left(\frac{1+ \sqrt{\lambda +\frac{\mu \epsilon}{2}}}{1- \sqrt{\frac{\mu}{2 \epsilon}+\nu}}\right)^{2}$, respectively.
	\end{lemma}
	
	\begin{proof}
		We prove this in the following two steps:
		
		\textbf{Step (I):} In this step, we show that $\{g_{n}\}_{n=0}^{\infty}$  is a Bessel sequence.
		
		For any finite scalar sequences $\{a_n\}_{n=0}^{\infty}$ and using inequality \eqref{P2_L3_1}, we compute
		\begin{align*}
			\norm{\sum_{n=0}^{l}a_{n}(f_{n}-g_{n})}^{2} & \leq  \lambda \norm{\sum_{n=0}^{l}a_{n}f_{n}}^{2} + \mu \norm{\sum_{n=0}^{l}a_{n}f_{n}} \, \norm{\sum_{n=0}^{l}a_{n}g_{n}} + \nu \norm{\sum_{n=0}^{l}a_{n}g_{n}}^{2}\\
			& \leq  \lambda \norm{\sum_{n=0}^{l}a_{n}f_{n}}^{2} +  \frac{\mu \epsilon}{2}
			\norm{\sum_{n=0}^{l}a_{n}f_{n}}^{2} + \frac{\mu}{2 \epsilon}
			\norm{\sum_{n=0}^{l}a_{n}g_{n}}^{2} + \nu \norm{\sum_{n=0}^{l}a_{n}g_{n}}^{2}\\
			\intertext{(using Young's Inequality given in Theorem \ref{Young})}
			& = \left(\lambda + \frac{\mu \epsilon}{2}\right) \norm{\sum_{n=0}^{l}a_{n}f_{n}}^{2} + \left(\frac{\mu}{2 \epsilon}+ \nu\right) \norm{\sum_{n=0}^{l}a_{n}g_{n}}^{2} \\
			& \leq \left(\sqrt{\lambda + \frac{\mu \epsilon}{2}} \, \norm{\sum_{n=0}^{l}a_{n}f_{n}} + \sqrt{\frac{\mu}{2 \epsilon}+ \nu} \, \norm{\sum_{n=0}^{l}a_{n}g_{n}} \right)^{2}
		\end{align*}
		that gives
		\begin{align*}
			\norm{\sum_{n=0}^{l}a_{n}(f_{n}-g_{n})} &  \leq \sqrt{\lambda + \frac{\mu \epsilon}{2}} \, \norm{\sum_{n=0}^{l}a_{n}f_{n}}+ \sqrt{\frac{\mu}{2 \epsilon}+ \nu} \, \norm{\sum_{n=0}^{l}a_{n}g_{n}}.
		\end{align*}
		Now,
		\begin{align*}
			\norm{\sum_{n=0}^{l}a_{n}g_{n}} & \leq
			\norm{\sum_{n=0}^{l}a_{n}f_{n}}+\norm{\sum_{n=0}^{l}a_{n}(f_{n}-g_{n})} \\
			& \leq \norm{\sum_{n=0}^{l}a_{n}f_{n}} + \sqrt{\lambda + \frac{\mu \epsilon}{2}} \, \norm{\sum_{n=0}^{l}a_{n}f_{n}}+ \sqrt{\frac{\mu}{2 \epsilon}+ \nu} \, \norm{\sum_{n=0}^{l}a_{n}g_{n}} \\
			& = \left(1 +  \sqrt{\lambda + \frac{\mu \epsilon}{2}} \right) \norm{\sum_{n=0}^{l}a_{n}f_{n}} + \sqrt{\frac{\mu}{2 \epsilon}+ \nu} \, \norm{\sum_{n=0}^{l}a_{n}g_{n}}.
		\end{align*}
		This implies that
		\begin{align*}
			\left(1-\sqrt{\frac{\mu}{2 \epsilon}+ \nu} \right)  \, \norm{\sum_{n=0}^{l}a_{n}g_{n}} & \leq \left(1 +  \sqrt{\lambda + \frac{\mu \epsilon}{2}} \right) \norm{\sum_{n=0}^{l}a_{n}f_{n}}.
		\end{align*}
		Since $\max{\{\lambda+ \frac{\mu \epsilon}{2}, \frac{\mu}{2 \epsilon}+ \nu\}}<1$, above inequality becomes
		\begin{align*}
			\norm{\sum_{n=0}^{l}a_{n}g_{n}} & \leq \left(\frac{1+ \sqrt{\lambda +\frac{\mu \epsilon}{2}}}{1- \sqrt{\frac{\mu}{2 \epsilon}+\nu}}\right)
			\norm{\sum_{n=0}^{l}a_{n}f_{n}}
		\end{align*}
		for all finite scalar sequences $\{a_n\}_{n=0}^{\infty}$. Thus, by Theorem \ref{P2_Th2.5}, $\{g_{n}\}_{n=0}^{\infty}$  is a Bessel sequence with Bessel bound $ \beta \left(\frac{1+ \sqrt{\lambda +\frac{\mu \epsilon}{2}}}{1- \sqrt{\frac{\mu}{2 \epsilon}+\nu}}\right)^{2}$ and its pre-frame operator  $T_{g}:\ell^2(\Z^{+}) \rightarrow \mathcal{H}$ given by
		$T_{g}\{a_{n}\}_{n = 0}^{\infty}=\sum_{n = 0}^{\infty}a_{n}g_{n}$,
		is well-defined and bounded.
		
		\textbf{Step (II):}
		In this step, we prove the lower frame condition for $\{g_{n}\}_{n=0}^{\infty}$. Let $T^{\dagger}_{f}:\mathcal{H} \rightarrow \ell^2(\Z^{+})$ be the pseudo-inverse of  $T_{f}$. Then, for any $f \in \mathcal{H}$, we have
		\begin{align*}
			T^{\dagger}_{f}f :=T^{*}_{f}(T_{f}T^{*}_{f})^{-1}f = \{ \langle f, (T_{f}T^{*}_{f})^{-1}f_{n} \rangle \}_{n=0}^{\infty} \ \big(\text{by Theorem \ref{P2_Th2.6}}).
		\end{align*}
		Note that if inequality \eqref{P2_L3_1} holds for any finite sequence of scalars, then it holds for any  sequence $\{a_n\}_{n=0}^{\infty} \in \ell^2(\mathbb{Z}^{+})$.  By invoking \eqref{P2_L3_1} for $\{a_{n}\}_{n=0}^{\infty} = 	T^{\dagger}_{f}f$, we get
		\begin{align*}
			&\norm{\sum_{n=0}^{\infty}\langle f, (T_{f}T^{*}_{f})^{-1}f_{n} \rangle (f_{n}-g_{n})}^{2}\\
			& \leq \lambda \norm{\sum_{n=0}^{\infty}\langle f, (T_{f}T^{*}_{f})^{-1}f_{n} \rangle f_{n}}^{2}  \\
			& \quad +  \mu \norm{\sum_{n=0}^{\infty}\langle f, (T_{f}T^{*}_{f})^{-1}f_{n} \rangle f_{n}} \, \norm{\sum_{n=0}^{\infty}\langle f, (T_{f}T^{*}_{f})^{-1}f_{n} \rangle g_{n}}\\
			&	 \quad + \nu \norm{\sum_{n=0}^{\infty}\langle f, (T_{f}T^{*}_{f})^{-1}f_{n} \rangle g_{n}}^{2}.
		\end{align*}
		That is,
		\begin{align*}
			\norm{f-T_{g}T^{\dagger}_{f}f}^{2} &\leq \lambda \norm{f}^{2} + \mu \norm{f} \, \norm{T_{g}T^{\dagger}_{f}f} + \nu \norm{T_{g}T^{\dagger}_{f}f}^{2}\\
			& \leq \lambda \norm{f}^{2} + \frac{\mu \epsilon}{2}\norm{f}^{2} + \frac{\mu}{2 \epsilon} \norm{T_{g}T^{\dagger}_{f}f}^{2}  + \nu \norm{T_{g}T^{\dagger}_{f}f}^{2} \ \big(\text{using Theorem \ref{Young}}\big)\\
			& = \left(\lambda + \frac{\mu \epsilon}{2}\right)\norm{f}^{2} + \left(\frac{\mu}{2 \epsilon}+ \nu\right)\norm{T_{g}T^{\dagger}_{f}f}^{2} \\
			& \leq \left(\sqrt{\lambda + \frac{\mu \epsilon}{2}} \, \norm{f} + \sqrt{\frac{\mu}{2 \epsilon}+ \nu} \, \norm{T_{g}T^{\dagger}_{f}f} \right)^{2} \ \text{for all} \ f \in \mathcal{H}.
		\end{align*}
		This implies that $\norm{f-T_{g}T^{\dagger}_{f}f} \leq \sqrt{\lambda + \frac{\mu \epsilon}{2}} \, \norm{f} + \sqrt{\frac{\mu}{2 \epsilon}+ \nu} \, \norm{T_{g}T^{\dagger}_{f}f} \ \text{for all} \ f \in \mathcal{H}$.
		As $\max{\{\lambda+ \frac{\mu \epsilon}{2}, \frac{\mu}{2 \epsilon}+ \nu\}}<1$, by  Lemma \ref{P2_L2.5}, $T_{g}T^{\dagger}_{f}$ is invertible, and
		$
		\norm{(T_{g}T^{\dagger}_{f})^{-1}} \leq \frac{1+\sqrt{\frac{\mu}{2 \epsilon}+ \nu}}{1- \sqrt{\lambda + \frac{\mu \epsilon}{2}}}.
		$
		Now, every $ f \in \mathcal{H}$ can be written as follows:
		\begin{align*}
			f=T_{g}T^{\dagger}_{f}(T_{g}T^{\dagger}_{f})^{-1}f = \sum_{n=0}^{\infty} \langle (T_{g}T^{\dagger}_{f})^{-1}f, (T_{f}T^{*}_{f})^{-1}f_{n} \rangle g_{n}.
		\end{align*}
		Therefore, for every $ f \in \mathcal{H}$, we have
		\begin{align*}
			\norm{f}^{4} = {\langle f,f \rangle}^{2} & = \abs{\sum_{n=0}^{\infty} \langle (T_{g}T^{\dagger}_{f})^{-1}f, (T_{f}T^{*}_{f})^{-1}f_{n} \rangle  \langle g_{n}, f \rangle}^{2} \\
			& \leq \sum_{n=0}^{\infty} \abs{\langle (T_{g}T^{\dagger}_{f})^{-1}f, (T_{f}T^{*}_{f})^{-1}f_{n} \rangle }^{2}
			\sum_{n=0}^{\infty} \abs{\langle g_{n},f \rangle }^{2} \\
			& \leq \frac{1}{\alpha}  \norm{(T_{g}T^{\dagger}_{f})^{-1}f}^{2} \sum_{n=0}^{\infty} \abs{\langle g_{n},f \rangle }^{2} \\
			& \leq \frac{1}{\alpha}\left(\frac{1+\sqrt{\frac{\mu}{2 \epsilon}+ \nu}}{1- \sqrt{\lambda + \frac{\mu \epsilon}{2}}}\right)^{2} \norm{f}^{2} \sum_{n=0}^{\infty} \abs{\langle g_{n},f \rangle }^{2}.
		\end{align*}
		This implies that
		\begin{align*}
			\sum_{n=0}^{\infty} \abs{\langle g_{n},f \rangle }^{2} \geq
			\alpha \left(\frac{1- \sqrt{\lambda + \frac{\mu \epsilon}{2}}}{1+\sqrt{\frac{\mu}{2 \epsilon}+ \nu}}\right)^{2} \norm{f}^{2} \,\, \text{for all} \,\, f \in \mathcal{H},
		\end{align*}
		This gives the lower frame condition for $\{g_{n}\}_{n=0}^{\infty}$. Hence, $\{g_{n}\}_{n=0}^{\infty}$  is a frame for $\mathcal{H}$ with  frame bounds $\alpha\Big(\frac{1-\sqrt{\lambda +\frac{\mu \epsilon}{2}}}{1+ \sqrt{\frac{\mu}{2 \epsilon}+\nu}}\Big)^{2}$ and $ \beta \left(\frac{1+ \sqrt{\lambda +\frac{\mu \epsilon}{2}}}{1- \sqrt{\frac{\mu}{2 \epsilon}+\nu}}\right)^{2}$, respectively.
		
	\end{proof}
	
We record several key observations regarding Lemma \ref{L_3.1}, detailing its connection with existing literature, discriminant conditions, and the role of the cross-term coefficients. These are given in the following remark:
		
	\begin{remark}\label{rem:Nagy_lemma_remarks}
		\begin{enumerate}[label=(\roman*)]
			\item \textbf{Connection with a perturbation result in  \cite[Theorem 2.1]{POATH}:}
			In Definition \ref{concls009} \eqref{Nagy_Def}, Retherford \cite{BSPWC} introduced the Nagy condition to study the invariance of completeness and basic sequences in normed spaces, which does not directly yield frame bounds or the invertibility of the frame operator. In Lemma \ref{L_3.1} , applying the $\epsilon$-version of Young's inequality (Theorem  \ref{Young} ) to the mixed term $\mu\left\|\sum_{n=0}^{l} a_n f_n\right\|\left\|\sum_{n=0}^{l} a_n g_n\right\|$ in \eqref{P2_L3_1} reduces the quadratic Nagy inequality to the linear synthesis-operator estimate
			\begin{align}\label{eq:linear_nagy_reduction}
				\left\|\sum_{n=0}^{l} a_n (f_n - g_n)\right\| \le \sqrt{\lambda + \frac{\mu\epsilon}{2}}\left\|\sum_{n=0}^{l} a_n f_n\right\| + \sqrt{\frac{\mu}{2\epsilon} + \nu}\left\|\sum_{n=0}^{l} a_n g_n\right\|.
			\end{align}
			Under the condition $\max\left\{\lambda + \frac{\mu\epsilon}{2}, \frac{\mu}{2\epsilon} + \nu\right\} < 1$, inequality (\ref{eq:linear_nagy_reduction}) satisfies the hypotheses of perturbation theorem in \cite[Theorem 2.1]{POATH} (with $\lambda_1 = \sqrt{\lambda + \frac{\mu\epsilon}{2}}$, $\lambda_2 = \sqrt{\frac{\mu}{2\epsilon} + \nu}$, and $\mu = 0$ in their notation).
			
			\item \label{rm_3.2(ii)} \textbf{Equivalence with the discriminant condition:} 	For $\lambda, \nu \in [0,1)$ and $\mu > 0$, a parameter $\epsilon > 0$ satisfies $\max\left\{\lambda + \frac{\mu\epsilon}{2}, \frac{\mu}{2\epsilon} + \nu\right\} < 1$ if and only if
			$$
			\frac{\mu}{2(1-\nu)} < \epsilon < \frac{2(1-\lambda)}{\mu},
			$$
			and such an $\epsilon > 0$ exists if and only if the discriminant-type condition
			$$
			\mu^2 < 4(1-\lambda)(1-\nu)
			$$
			holds. In particular, Retherford's condition $\mu^2 \le (1-\lambda)(1-\nu)$ in Definition  \ref{concls009} \eqref{Nagy_Def} is a stronger sufficient condition that automatically guarantees the existence of $\epsilon > 0$ in Lemma \ref{L_3.1} (taking $\epsilon = \sqrt{\frac{1-\lambda}{1-\nu}}$ yields $\max\left\{\lambda + \frac{\mu\epsilon}{2}, \frac{\mu}{2\epsilon} + \nu\right\} \le \max\left\{\frac{1+\lambda}{2}, \frac{1+\nu}{2}\right\} < 1$).

			\item  \textbf	{Significance of the mixed term $\mu\left\|\sum a_n f_n\right\|\left\|\sum a_n g_n\right\|$) and role of $\epsilon > 0$:}\\
			Although the Nagy condition \eqref{P2_L3_1} reduces via Young's inequality to (\ref{eq:linear_nagy_reduction}), retaining the mixed term with coefficient $\mu$ in Lemma \ref{L_3.1} offers three advantages:
			\begin{enumerate}
				\item[(a)] In Hilbert spaces, squared-norm expansions of perturbations of the form $f_n - g_n $ naturally produce cross-terms $\left\|\sum a_n f_n\right\|\left\|\sum a_n g_n\right\|$. Eliminating the mixed term  via $(a+b)^2 \le 2a^2 + 2b^2$ would double the constants $\lambda$ and $\nu$ and halve the admissible parameter range.
				
				\item[(b)] Whenever $0 \le \mu \le 2\sqrt{\lambda\nu}$, the right-hand side of \eqref{P2_L3_1}  is bounded above by $\left(\sqrt{\lambda}\left\|\sum a_n f_n\right\| + \sqrt{\nu}\left\|\sum a_n g_n\right\|\right)^2$, so the mixed term is absorbed with no loss in the frame bounds (giving $\alpha\left(\frac{1-\sqrt{\lambda}}{1+\sqrt{\nu}}\right)^2$ and $\beta\left(\frac{1+\sqrt{\lambda}}{1-\sqrt{\nu}}\right)^2$).
				
				\item[(c)] When $\mu > 2\sqrt{\lambda\nu}$, varying $\epsilon$ over the open interval $\left(\frac{\mu}{2(1-\nu)}, \frac{2(1-\lambda)}{\mu}\right)$ provides a continuous one-parameter family of valid frame bounds, allowing one to tighten either the lower frame bound or the upper frame bound depending on the application.
				
			\end{enumerate}
			
			\item \textbf{Alternative upper frame bounds:}
			If one imposes the additional quadratic condition $2(1-\nu)\epsilon^2 - (\mu+2\nu)\epsilon - \mu > 0$ alongside $\max\left\{\lambda + \frac{\mu\epsilon}{2}, \frac{\mu}{2\epsilon} + \nu\right\} < 1$, then by squaring on both side of $
			\norm{\sum_{n=0}^{l}a_{n}g_{n}}  \leq
			\norm{\sum_{n=0}^{l}a_{n}f_{n}}+\norm{\sum_{n=0}^{l}a_{n}(f_{n}-g_{n})}
			$, and the repeated use of Young's inequality, the upper frame bound given in Lemma \ref{L_3.1} can be replaced by the alternative expression
			$$\beta\left(\frac{(\mu+2) \epsilon^{3}+[2(1+\lambda)+\mu]\epsilon^{2}+2\lambda  \epsilon}{2(1-\nu) \epsilon^{2}- (\mu +2 \nu)\epsilon-\mu }\right).$$
			
			Similarly, replacing the quadratic condition by the linear condition $(1-2\nu)\epsilon - \mu > 0$ yields the alternative upper bound $\beta\left(\frac{\mu\epsilon^2 + 2(1+\lambda)\epsilon}{(1-2\nu)\epsilon - \mu}\right)$. Depending on the parameters, these alternative bounds can occasionally be smaller than the one derived from (\ref{eq:linear_nagy_reduction}).
			
		\end{enumerate}
	\end{remark}
	
Building upon this general framework, we now establish sufficient conditions for the stable recovery of the initial state under simultaneous perturbations of the sampling vector and the evolution operator.
	
	\begin{theorem}\label{P2_Th3.9'}
		Let $g$, $h \in \mathcal{H}$ and let $A$ and $B$ be  bounded linear operators acting on $\mathcal{H}$. Suppose that
		\begin{enumerate}[label=(\roman*)]
			\item 	any $f \in \mathcal{H}$ is recovered from the samples $\{\langle A^{n}f,g\rangle\}_{n \geq 0}$ in a stable way and there exist constants $\lambda$, $\mu$, $\nu  \geq 0$ such that
			\begin{align}
				\notag
				\norm{\sum_{n=0}^{l}a_{n}((A^{*})^{n}g-(B^{*})^{n}h)}^{2} & \leq \lambda \norm{\sum_{n=0}^{l}a_{n}(A^{*})^{n}g}^{2}  +  \mu \norm{\sum_{n=0}^{l}a_{n}(A^{*})^{n}g} \, \norm{\sum_{n=0}^{l}a_{n}(B^{*})^{n}h} \\
				& \quad + \nu \norm{\sum_{n=0}^{l}a_{n}(B^{*})^{n}h}^{2}, \;\; l \in \Z^{+},
				\label{P2_S1_5'}
			\end{align}
			for all finite scalar sequences $\{a_n\}_{n=0}^{\infty}$,
			\item $\max{\{\lambda+ \frac{\mu \epsilon}{2}, \frac{\mu}{2 \epsilon}+ \nu\}}<1$, where $\epsilon >0$.
\end{enumerate}
		Then, any $f \in \mathcal{H}$ can also be recovered from the samples $\{\langle B^{n}f,h\rangle\}_{n \geq 0}$ in a stable way.
	\end{theorem}
	
\begin{proof}
Let $f_{n}=(A^{*})^{n}g$ and $g_{n}=(B^{*})^{n}h$.
		Then, by Proposition \ref{P2_P2.3}, the stable recovery of any $f \in \mathcal{H}$ from the dynamical samples $\{\langle A^{n}f,g\rangle\}_{n \geq 0}$ is equivalent to the system $\{(A^{*})^{n}g\}_{n \geq 0}$ forming a frame for $\mathcal{H}$. Since the perturbation condition \eqref{P2_S1_5'} assumed in the Theorem \ref{P2_Th3.9'} matches inequality \eqref{P2_L3_1} of Lemma \ref{L_3.1} for the sequences $\{f_{n}\}_{n=0}^{\infty}$ and $\{g_{n}\}_{n=0}^{\infty}$, an application of Lemma \ref{L_3.1} implies that the perturbed system $\{(B^{*})^{n}h\}_{n \geq 0}$ is also a frame for $\mathcal{H}$. Finally, invoking Proposition \ref{P2_P2.3} for the perturbed system, we conclude that any $f \in \mathcal{H}$ can also be recovered from the perturbed dynamical samples $\{\langle B^{n}f,h\rangle\}_{n \geq 0}$ in a stable way.
	\end{proof}
	
	An example illustrating Theorem \ref{P2_Th3.9'} is given below.
	
	\begin{example}\label{P2_Exp3.6'}
		Let $\mathcal{H}=\ell^2(\Z^{+})$ and let $\{e_{n}\}_{n \geq 0}$ be  an orthonormal basis of $\mathcal{H}$. Consider the left shift operator $A: \ell^2(\Z^{+}) \rightarrow \ell^2(\Z^{+})$ defined by $A(x_0, x_1,x_2, \ldots)=(x_1,x_2,x_3, \ldots)$. Then, its Hilbert-adjoint operator $A^{*}:\ell^2(\Z^{+}) \rightarrow \ell^2(\Z^{+})$ is $A^{*}(x_0, x_1,x_2, \ldots)=(0,x_0, x_1,x_2, \ldots)$, or $A^{*}e_{n}=e_{n+1} \;\; \text{for all} \;\; n \geq 0$. Let $g=e_{0}$ and $h=\frac{e_{0}}{2}$. Then,  $\{(A^{*})^{n}g\}_{n \geq 0}=\{e_{n}\}_{n \geq 0}$, which is an orthonormal basis of the space $\ell^2(\Z^{+})$. Hence, every $f \in \mathcal{H}$ is stably recovered from the samples $\{\langle A^{n}f,g\rangle\}_{n \geq 0}$.
		
		Consider the rank-one perturbation  $B:  \ell^2(\Z^{+}) \rightarrow \ell^2(\Z^{+})$  of the operator $A$, that is, $B=A+K$, where $K: \ell^2(\Z^{+}) \rightarrow \ell^2(\Z^{+})$ defined as  $K(x)=\alpha \langle x, e_{0} \rangle  e_{0}$ and $ 0< \alpha <1 $ is fixed. Thus, $B(x_0, x_1,x_2, \ldots)=(x_1+ \alpha x_{0},x_2,x_3, \ldots)$, or $B(e_{0})=\alpha e_{0}$ and $ B(e_{n})=e_{n-1} \;\; \text{for all} \;\; n \geq 1$. The Hilbert-adjoint operator of $B$ is $B^{*}:\ell^2(\Z^{+}) \rightarrow \ell^2(\Z^{+})$ given by $B^{*}(x_0, x_1,x_2, \ldots)=(\alpha x_{0},x_0, x_1,x_2, \ldots)$, or $B^{*}e_{0}=e_{1}+\alpha e_{0}$, and $ B^{*}e_{n}=e_{n+1} \;\; \text{for all} \;\; n \geq 1$. Therefore,
		\begin{align*}
			&\{(B^{*})^{n}h\}_{n \geq 0}=\left\{\frac{1}{2}(e_{n}+\alpha e_{n-1}+\alpha^2 e_{n-2}+ \ldots + \alpha^ne_{0})\right\}_{n \geq 0}, \intertext{and}
			& \{(A^{*})^{n}g-(B^{*})^{n}h\}_{n \geq 0}=\left\{\frac{1}{2}e_{n}-\frac{1}{2}(\alpha e_{n-1}+\alpha^2 e_{n-2}+ \ldots + \alpha^ne_{0})\right\}_{n \geq 0}.
		\end{align*}
		Now, for any finite scalar sequence $\{a_n\}_{n=0}^{\infty}$, we compute
		\begin{align}
			\norm{\sum_{n=0}^{l}a_{n}(A^{*})^{n}g}^{2}  & =\norm{\sum_{n=0}^{l}a_{n}e_{n}}^{2} =\sum_{n=0}^{l}\abs{a_n}^{2}; \label{lkex771'} \\
			\norm{\sum_{n=0}^{l}a_{n}(B^{*})^{n}h}^{2} & =\norm{\frac{1}{2}\sum_{n=0}^{l}a_{n}(e_{n}+\alpha e_{n-1}+\alpha^2 e_{n-2}+ \ldots +  \alpha^ne_{0})}^{2} \notag\\
			& = \frac{1}{4} \norm{a_0e_{0}+a_1(e_{1}+\alpha e_{0})+\ldots+a_l(e_{l}+\alpha e_{l-1}+\ldots +\alpha^{l}e_{0})}^{2} \notag\\
			&= \frac{1}{4} \norm{(a_0+a_1 \alpha+\ldots+a_l \alpha^{l})e_{0}+(a_1+a_2 \alpha+\ldots+a_l \alpha^{l-1})e_{1}+ \ldots+a_l e_{l}}^{2} \notag\\
			&=  \frac{1}{4} \left(\abs{a_0+a_1 \alpha +\ldots+a_l \alpha^{l}}^{2} + \abs{a_1+a_2 \alpha +\ldots+a_l \alpha^{l-1}}^{2} \right. \notag\\
			& \quad \left. + \ldots + \abs{a_{l-1}+a_l \alpha }^{2} + \abs{a_l}^{2} \right) \notag\\
			&= \frac{1}{4} \left[ \abs{\sum_{m=0}^{l}a_m \alpha^{m}}^{2} +\abs{\sum_{m=1}^{l}a_m \alpha^{m-1}}^{2} + \ldots+\abs{\sum_{m=l}^{l}a_m \alpha^{m-l}}^{2} \right] \notag\\
			&= \frac{1}{4} \sum_{n=0}^{l}\abs{\sum_{m=n}^{l}a_m \alpha^{m-n}}^{2}; \label{lkex772'}
		\end{align}
		and
		\begin{align}\label{lkex7733'}
			\norm{\sum_{n=0}^{l}a_{n}((A^{*})^{n} g-(B^{*})^{n} h)}^{2}  &=\norm{\frac{1}{2} \sum_{n=0}^{l}a_{n}e_{l}-\frac{1}{2}\sum_{n=1}^{l}a_{n}(\alpha e_{n-1}+\alpha^2 e_{n-2}+ \ldots + \alpha^ne_{0})}^{2} \notag\\
			&= \frac{1}{4} \left \|(a_0+a_1+ \ldots + a_l)-(a_1(\alpha e_{0})+a_2(\alpha e_{1}+\alpha^{2} e_{0}) \right. \notag\\
			& \quad \left. +\ldots +a_l(\alpha e_{l-1}+ \alpha^{2} e_{l-2}+\ldots +\alpha^{l}e_{0}))\right \|^{2} \notag\\
			&= \frac{1}{4} \left \| (a_0-(a_1 \alpha + a_2 \alpha^{2}+\ldots+a_l \alpha^{l}))e_{0} \right. \notag \\
			& \quad \left. +(a_1-(a_2 \alpha+\ldots+a_l \alpha^{l-1}))e_{1} + \ldots+(a_{l-1}-a_l \alpha )e_{l-1} + a_l e_{l}\right \|^{2} \notag\\
			&= \frac{1}{4} \left(\abs{a_0-a_1 \alpha -\ldots-a_l \alpha^{l}}^{2} + \abs{a_1-a_2 \alpha -\ldots-a_l \alpha^{l-1}}^{2} \right.  \notag\\
			& \quad \left. + \ldots + \abs{a_{l-1}-a_l \alpha}^{2}+ \abs{a_l}^{2} \right) \notag\\
			&=\frac{1}{4} \left[ \abs{a_0-\sum_{m=1}^{l}a_m \alpha^{m}}^{2} +\abs{a_1-\sum_{m=2}^{l}a_m \alpha^{m-1}}^{2} \right. \notag\\
			& \quad \left. + \ldots+\abs{a_{l-1}-\sum_{m=l}^{l}a_m \alpha^{m-(l-1)}}^{2} +\abs{a_l}^{2}\right] \notag\\
			&= \frac{1}{4} \sum_{n=0}^{l-1}\abs{a_n-\sum_{m=n+1}^{l}a_m \alpha^{m-n}}^{2} + \frac{1}{4}\abs{a_l}^{2}.
		\end{align}
		
		Write $M_{n}=\sum_{m=n+1}^{l}a_m \alpha^{m-n}$ and $T_{n}=\sum_{m=n}^{l}a_m \alpha^{m-n}$. Then,
		\begin{align*}
			M_{n}=\sum_{m=n+1}^{l}a_m \alpha^{m-n}=\alpha \left(\sum_{m=n+1}^{l}a_m \alpha^{m-(n+1)}\right) = \alpha T_{n+1}.
		\end{align*}
		Therefore,
		\begin{align}\label{lkex7744'}
			\sum_{n=0}^{l-1} \abs{M_{n}}^{2} = \sum_{n=0}^{l-1} \abs{\alpha T_{n+1}}^{2} = \abs{\alpha}^{2} \sum_{n=0}^{l-1} \abs{ T_{n+1}}^{2} = \abs{\alpha}^{2} \sum_{n=1}^{l} \abs{ T_{n}}^{2} \leq \abs{\alpha}^{2} \sum_{n=0}^{l} \abs{ T_{n}}^{2}.
		\end{align}
		
		Now, using $\abs{a_{n}-M_{n}}^{2}\leq \abs{a_{n}}^{2}+2\abs{a_{n}}\abs{M_{n}}+\abs{M_{n}}^{2}$, the Cauchy--Schwarz inequality, and equations \eqref{lkex771'},  \eqref{lkex772'},  \eqref{lkex7733'} and \eqref{lkex7744'}, we get
		\begin{align*}
			\norm{\sum_{n=0}^{l}a_{n}((A^{*})^{n} g-(B^{*})^{n} h)}^{2}  &= \frac{1}{4} \sum_{n=0}^{l-1}\abs{a_n- M_{n}}^{2} + \frac{1}{4}\abs{a_l}^{2}\\
			& \leq \frac{1}{4}\sum_{n=0}^{l}\abs{a_n}^{2}+\frac{1}{2} \sum_{n=0}^{l}\abs{a_n}\abs{M_{n}}+\frac{1}{4}	\sum_{n=0}^{l-1} \abs{M_{n}}^{2} \\
			& \leq \frac{1}{4}\sum_{n=0}^{l}\abs{a_n}^{2}+\frac{1}{2}\left(\sum_{n=0}^{l}\abs{a_n}^{2}\right)^{\frac{1}{2}}\left(	\sum_{n=0}^{l-1} \abs{M_{n}}^{2} \right)^{\frac{1}{2}}+\frac{|\alpha|^{2}}{4} \sum_{n=1}^{l} \abs{ T_{n}}^{2} \\
			& \leq \frac{1}{4} \left(\sum_{n=0}^{l}\abs{a_n}^{2}\right)+ |\alpha| \left(\sum_{n=0}^{l}\abs{a_n}^{2}\right) ^{\frac{1}{2}}\left(\frac{1}{4}\sum_{n=0}^{l}\abs{\sum_{m=n}^{l}a_m \alpha^{m-n}}^{2}\right)^{\frac{1}{2}} \\
			& + |\alpha|^{2} \left( \frac{1}{4}\sum_{n=0}^{l}\abs{\sum_{m=n}^{l}a_m \alpha^{m-n}}^{2} \right)\\
			& = \lambda \norm{\sum_{n=0}^{l}a_{n}(A^{*})^{n}g}^{2} + \mu \norm{\sum_{n=0}^{l}a_{n}(A^{*})^{n}g} \, \norm{\sum_{n=0}^{l}a_{n}(B^{*})^{n}h} \\
			&	\quad + \nu \norm{\sum_{n=0}^{l}a_{n}(B^{*})^{n}h}^{2},
		\end{align*}
		where we choose $\alpha=0.5$, $\lambda=0.25$, $\mu=|\alpha|=0.5$, and $\nu=|\alpha|^{2}=0.25$.
		Thus, hypothesis  \eqref{P2_S1_5'} of Theorem \ref{P2_Th3.9'} is satisfied.	Also, for any $\epsilon \in (\frac{1}{3},3)$, we have  $\max{\{\lambda+ \frac{\mu \epsilon}{2}, \frac{\mu}{2 \epsilon}+ \nu\}}<1$. Hence, by Theorem \ref{P2_Th3.9'}, any $f \in \mathcal{H}$ can be stably recovered from the samples $\{\langle B^{n}f,h\rangle\}_{n \geq 0}$.
	\end{example}
	
	As an immediate consequence of Theorem \ref{P2_Th3.9'}, the following corollary records the special cases in which only the sampling vector or  the evolution operator is perturbed.
\begin{corollary}\label{cor_Nagy}
	Let $A, B \in \mathcal{B}(\mathcal{H})$ and $g, h \in \mathcal{H}$ be such that any $f \in \mathcal{H}$ is recovered from the samples $\{\langle A^{n}f,g\rangle\}_{n \geq 0}$ in a stable way.  Suppose there exist  real numbers $\lambda$, $\mu$, $\nu  \geq 0$ and $\epsilon > 0$ such that  $\max{\left\{\lambda+ \frac{\mu \epsilon}{2}, \frac{\mu}{2 \epsilon}+ \nu\right\}}<1$.
	\begin{enumerate}[label=(\roman*)]
		\item  \label{cor_Nagy_1} If for all $l \in \Z^{+}$  and all finite scalar sequences $\{a_n\}_{n=0}^{\infty}$,
		\begin{align*}
			\norm{\sum_{n=0}^{l}a_{n}(A^{*})^{n}(g-h)}^{2} & \leq  \lambda \norm{\sum_{n=0}^{l}a_{n}(A^{*})^{n}g}^{2} + \mu \norm{\sum_{n=0}^{l}a_{n}(A^{*})^{n}g} \, \norm{\sum_{n=0}^{l}a_{n}(A^{*})^{n}h}\\
			& \quad + \nu \norm{\sum_{n=0}^{l}a_{n}(A^{*})^{n}h}^{2},
		\end{align*}
		then any $f \in \mathcal{H}$ can also be recovered from the samples $\{\langle A^{n}f,h\rangle\}_{n \geq 0}$ in a stable way.
		
		\item \label{cor_Nagy_2} If for all $l \in \Z^{+}$  and all finite scalar sequences $\{a_n\}_{n=0}^{\infty}$,
		\begin{align*}
			\norm{\sum_{n=0}^{l}a_{n}((A^{*})^{n}-(B^{*})^{n})g}^{2} & \leq \lambda \norm{\sum_{n=0}^{l}a_{n}(A^{*})^{n}g}^{2} + \mu \norm{\sum_{n=0}^{l}a_{n}(A^{*})^{n}g} \, \norm{\sum_{n=0}^{l}a_{n}(B^{*})^{n}g} \\
			&	\quad + \nu \norm{\sum_{n=0}^{l}a_{n}(B^{*})^{n}g}^{2},
		\end{align*}
		then any $f \in \mathcal{H}$ can also be recovered from the samples $\{\langle B^{n}f,g\rangle\}_{n \geq 0}$ in a stable way.
	\end{enumerate}
\end{corollary}

\begin{proof}
	Parts \ref{cor_Nagy_1} and \ref{cor_Nagy_2} follow from Theorem \ref{P2_Th3.9'} by setting $B = A$ and $h = g$, respectively.
\end{proof}

	To account for scenarios where measurement errors and system modeling errors arise independently, we next present a split-sum formulation of simultaneous perturbation.
	\begin{theorem}\label{P2_Th3.9}
		Let $g$, $h \in \mathcal{H}$ and let $A$ and $B$ be  bounded linear operators acting on $\mathcal{H}$. Suppose that
		\begin{enumerate}[label=(\roman*)]
			\item 	any $f \in \mathcal{H}$ is recovered from the samples $\{\langle A^{n}f,g\rangle\}_{n \geq 0}$ in a stable way and there exist constants $\lambda$, $\mu$, $\nu  \geq 0$ such that
			\begin{align}
				\notag
				&	\norm{\sum_{n=0}^{l}a_{n}(A^{*})^{n}(g-h)}^{2} + \norm{\sum_{n =0}^{l}a_{n}((A^{*})^{n}-(B^{*})^{n})h}^{2} \leq \lambda \norm{\sum_{n=0}^{l}a_{n}(A^{*})^{n}g}^{2}  \\
				& \quad +  \mu \norm{\sum_{n=0}^{l}a_{n}(A^{*})^{n}g} \, \norm{\sum_{n=0}^{l}a_{n}(B^{*})^{n}h}+ \nu \norm{\sum_{n=0}^{l}a_{n}(B^{*})^{n}h}^{2}, \;\; l \in \Z^{+},
				\label{P2_S1_5}
			\end{align}
			for all finite scalar sequences $\{a_n\}_{n=0}^{\infty}$,
			\item $\max{\{2\lambda+ \mu \epsilon, \frac{\mu}{ \epsilon}+ 2\nu\}}<1$, where $\epsilon >0$.
		\end{enumerate}
		Then, any $f \in \mathcal{H}$ can also be recovered from the samples $\{\langle B^{n}f,h\rangle\}_{n \geq 0}$ in a stable way.
	\end{theorem}
	
		\begin{proof}
		By Proposition \ref{P2_P2.3}, the stable recovery of any $f \in \mathcal{H}$ from the dynamical samples $\{\langle A^{n}f,g\rangle\}_{n \geq 0}$ is equivalent to the system $\{(A^{*})^{n}g\}_{n \geq 0}$ forming a frame for $\mathcal{H}$.
		To verify that the perturbed system $\{(B^{*})^{n}h\}_{n \geq 0}$ satisfies the frame conditions under the split-sum Nagy-type hypothesis, we apply the triangle inequality, algebraic expansion $(a+b)^2 \le 2a^2 + 2b^2$, hypothesis \eqref{P2_S1_5}, and Young's inequality (Theorem \ref{Young}) :
		\begin{align*}
			\norm{\sum_{n=0}^{l}a_{n}((A^{*})^{n}g-(B^{*})^{n}h)}^{2} & \leq \left(\norm{\sum_{n=0}^{l}a_{n}(A^{*})^{n}(g-h)} + \norm{\sum_{n =0}^{l}a_{n}((A^{*})^{n}-(B^{*})^{n})h}\right)^{2} \notag\\
			& \leq 2 \left[\norm{\sum_{n=0}^{l}a_{n}(A^{*})^{n}(g-h)}^{2} + \norm{\sum_{n =0}^{l}a_{n}((A^{*})^{n}-(B^{*})^{n})h}^{2}\right] \notag\\
			& \leq 2\lambda \norm{\sum_{n=0}^{l}a_{n}(A^{*})^{n}g}^{2}  + 2 \mu \norm{\sum_{n=0}^{l}a_{n}(A^{*})^{n}g} \, \norm{\sum_{n=0}^{l}a_{n}(B^{*})^{n}h} \notag\\
			&  \quad + 2\nu \norm{\sum_{n=0}^{l}a_{n}(B^{*})^{n}h}^{2} \quad \big(\text{by using \eqref{P2_S1_5}}\big) \notag \\
			& \leq 2\lambda \norm{\sum_{n=0}^{l}a_{n}(A^{*})^{n}g}^{2} + \mu \epsilon \norm{\sum_{n=0}^{l}a_{n}(A^{*})^{n}g}^{2} + \frac{\mu}{\epsilon} \norm{\sum_{n=0}^{l}a_{n}(B^{*})^{n}h}^{2} \notag \\
			& \quad + 2\nu \norm{\sum_{n=0}^{l}a_{n}(B^{*})^{n}h}^{2} \ (\text{by Theorem \ref{Young}}) \notag \\
			& = \left(2\lambda +\mu \epsilon\right)  \norm{\sum_{n=0}^{l}a_{n}(A^{*})^{n}g}^{2}  + \left(\frac{\mu}{\epsilon}+ 2\nu\right) \norm{\sum_{n=0}^{l}a_{n}(B^{*})^{n}h}^{2} \notag \\
			& \leq \left(\sqrt{2\lambda + \mu \epsilon} \, \norm{\sum_{n=0}^{l}a_{n}(A^{*})^{n}g} + \sqrt{\frac{\mu}{\epsilon}+ 2\nu} \, \norm{\sum_{n=0}^{l}a_{n}(B^{*})^{n}h} \right)^{2}
		\end{align*}
		which yields the linear reduction
		\begin{align*}
			\norm{\sum_{n=0}^{l}a_{n}((A^{*})^{n}g-(B^{*})^{n}h)} &  \leq \sqrt{2\lambda + \mu \epsilon} \, \norm{\sum_{n=0}^{l}a_{n}(A^{*})^{n}g} + \sqrt{\frac{\mu}{\epsilon}+ 2\nu} \, \norm{\sum_{n=0}^{l}a_{n}(B^{*})^{n}h}  .
		\end{align*}
		Since this linear estimate matches the framework of Lemma \ref{L_3.1}, it follows that $\{(B^{*})^{n}h\}_{n \geq 0}$ forms a frame for $\mathcal{H}$. Finally, invoking Proposition \ref{P2_P2.3} for the perturbed system, we conclude that any $f \in \mathcal{H}$ can also be recovered from the perturbed dynamical samples $\{\langle B^{n}f,h\rangle\}_{n \geq 0}$ in a stable way.
	\end{proof}
	
The following remark highlights the complementary nature of Theorem \ref{P2_Th3.9'} and Theorem \ref{P2_Th3.9}, contrasting their theoretical generality with their practical error-decoupling utility.
	
	\begin{remark}\label{rem:Nagy-comparison}
		While Theorem \ref{P2_Th3.9'} and Theorem \ref{P2_Th3.9} both establish stability under simultaneous perturbations of the sampling vector and the evolution operator, they serve complementary purposes: Theorem \ref{P2_Th3.9'} provides a direct and general theoretical formulation by imposing the Nagy-type condition on the net combined orbit difference $((A^{*})^{n}g-(B^{*})^{n}h)$, whereas Theorem \ref{P2_Th3.9} offers a practical, decoupled approach by separating measurement errors (sensor perturbations altering $g$ to $h$) from system modeling errors (operator perturbations altering $A$ to $B$). In applications, bounding these two error sources independently is often more tractable than estimating their combined difference directly, making Theorem \ref{P2_Th3.9} highly applicable to numerical verification.
	\end{remark}

	\section{Dynamical Samples using Pollard-Hilding-Type Frame Perturbation}\label{P2_sec4}
		In this section, we investigate the stability of dynamical sampling systems under Pollard–Hilding-type frame perturbations. To prepare for our initial-state recovery results, we first establish a basic frame lemma in Hilbert spaces, where Pollard–Hilding-type conditions allow us to derive explicit frame bounds for given any exponent $k>0$.
		
	\begin{lemma} \label{L_3.3}
		Let $\{f_{n}\}_{n=0}^{\infty}$ be a frame for $\mathcal{H}$ with bounds $\alpha$ and $\beta$. Let $\{g_{n}\}_{n=0}^{\infty}$ be any sequence in $\mathcal{H}$ and if for a given positive real number $k$, there exist constants $\lambda_{1}$, $\lambda_{2} \geq 0$ such that
		$\max{\{\lambda_{1}, \lambda_{2}\}} < \min{\{1, 2^{k-1}\}}$ and
		\begin{align}
			\norm{\sum_{n=0}^{l}a_{n}(f_{n}-g_{n})} \leq \left[\lambda_{1} \norm{\sum_{n=0}^{l}a_{n}f_{n}}^{k} + \lambda_{2} \norm{\sum_{n=0}^{l}a_{n}g_{n}}^{k}\right]^{\frac{1}{k}}, \;\; l \in \Z^{+},
			\label{P2_L2_1}
		\end{align}
		for all finite scalar sequences $\{a_n\}_{n=0}^{\infty}$. Then $\{g_{n}\}_{n=0}^{\infty}$  is a frame for $\mathcal{H}$ with frame bounds
\begin{align*}
\alpha\Big(\frac{1-c^{\frac{1}{k}}\lambda_{1}^{\frac{1}{k}}}{1+c^{\frac{1}{k}}\lambda_{2}^{\frac{1}{k}}}\Big)^{2} \quad \text{and} \quad \beta\Big(\frac{1+c^{\frac{1}{k}}\lambda_{1}^{\frac{1}{k}}}{1-c^{\frac{1}{k}}\lambda_{2}^{\frac{1}{k}}}\Big)^{2}, \quad \text{where} \quad  c=\begin{cases}
			1, & \;\text{if} \;  k \geq 1\\
			2^{1-k}, & \;  \text{if} \; k \leq 1 .
		\end{cases}
\end{align*}
	\end{lemma}
	
	\begin{proof}
		First, we show that $\{g_{n}\}_{n=0}^{\infty}$  satisfies the upper frame condition. Consider
		\begin{align*}
			\norm{\sum_{n=0}^{l}a_{n}g_{n}} & \leq  \norm{\sum_{n=0}^{l}a_{n}f_{n}}+	\norm{\sum_{n=0}^{l}a_{n}(f_{n}-g_{n})} \\
			& \leq \norm{\sum_{n=0}^{l}a_{n}f_{n}} + \left[\lambda_{1}
			\norm{\sum_{n=0}^{l}a_{n}f_{n}}^{k} + \lambda_{2}
			\norm{\sum_{n=0}^{l}a_{n}g_{n}}^{k}\right]^{\frac{1}{k}} \, \big(\text{using \eqref{P2_L2_1}}\big) \\
			& \leq \norm{\sum_{n=0}^{l}a_{n}f_{n}} + c^{\frac{1}{k}} \left[ \left[\lambda_{1}^{\frac{1}{k}}
			\norm{\sum_{n=0}^{l}a_{n}f_{n}} + \lambda_{2}^{\frac{1}{k}}
			\norm{\sum_{n=0}^{l}a_{n}g_{n}}\right]^{k}\right]^{\frac{1}{k}}.
		\end{align*}
		Note that the last step in above follows from Theorem \ref{Theorem 2.8}.
		This implies that
		\begin{align*}
			(1-c^{\frac{1}{k}} \lambda_{2}^{\frac{1}{k}})	\norm{\sum_{n=0}^{l}a_{n}g_{n}} & \leq (1+c^{\frac{1}{k}} \lambda_{1}^{\frac{1}{k}}) \norm{\sum_{n=0}^{l}a_{n}f_{n}}, \quad \text{where} \quad c=\begin{cases}
				1, & \;\text{if} \;  k \geq 1,\\
				2^{1-k}, & \;  \text{if} \; k \leq 1.
			\end{cases}
		\end{align*}
		That is,
		\begin{align*}
			\norm{\sum_{n=0}^{l}a_{n}g_{n}} \leq \frac{(1+c^{\frac{1}{k}} \lambda_{1}^{\frac{1}{k}})}{	(1-c^{\frac{1}{k}} \lambda_{2}^{\frac{1}{k}})}\norm{\sum_{n=0}^{l}a_{n}f_{n}} \ \Big(\text{as} \max{\{\lambda_{1}, \lambda_{2}\}} < \min{\{1, 2^{k-1}\}}\Big),
		\end{align*}
		for all finite scalar sequences $\{a_n\}_{n=0}^{\infty}$. This gives the upper frame condition for $\{g_{n}\}_{n=0}^{\infty}$ with upper bound $\beta\Big(\frac{1+c^{\frac{1}{k}}\lambda_{1}^{\frac{1}{k}}}{1-c^{\frac{1}{k}}\lambda_{2}^{\frac{1}{k}}}\Big)^{2}$, where $\beta$ is the upper frame bound of $\{f_{n}\}_{n=0}^{\infty}$.
		
		Now, we show that $\{g_{n}\}_{n=0}^{\infty}$ satisfies lower frame condition. Let $T_{f}$ and $T_{g}$ be the pre-frame operators of $\{f_{n}\}_{n=0}^{\infty}$ and $\{g_{n}\}_{n=0}^{\infty}$, respectively. Let $\alpha$ be a lower frame bound of $\{f_{n}\}_{n=0}^{\infty}$. Recall that the pseudo-inverse of $T_{f}$  is given by
		$ T^{\dagger}_{f} f = \{ \langle f, (T_{f}T^{*}_{f})^{-1}f_{n} \rangle \}_{n=0}^{\infty}$, $f \in \mathcal{H}$. Note that if inequality \eqref{P2_L2_1} holds for any finite sequence of scalars, then it holds for any  sequence $\{a_n\}_{n=0}^{\infty} \in \ell^2(\mathbb{Z}^{+})$.  By invoking \eqref{P2_L2_1} for $\{a_{n}\}_{n=0}^{\infty} = 	T^{\dagger}_{f}f$, we get
		
		\begin{align*}
			& \norm{\sum_{n=0}^{\infty}\langle f, (T_{f}T^{*}_{f})^{-1}f_{n} \rangle (f_{n}-g_{n})}  \\
			& \leq \left[\lambda_{1} \norm{\sum_{n=0}^{\infty}\langle f, (T_{f}T^{*}_{f})^{-1}f_{n} \rangle f_{n}}^{k}+ \lambda_{2} \norm{\sum_{n=0}^{\infty}\langle f, (T_{f}T^{*}_{f})^{-1}f_{n} \rangle g_{n}}^{k}\right]^{\frac{1}{k}}.
		\end{align*}
		Therefore, for every $f \in \mathcal{H} $, we have
		\begin{align*}
			\norm{f-T_{g}T^{\dagger}_{f}f}  \leq \left[\lambda_{1} \norm{f}^{k} + \lambda_{2} \norm{T_{g}T^{\dagger}_{f}f}^{k}\right]^{\frac{1}{k}}
			\leq c^{\frac{1}{k}} \lambda_{1}^{\frac{1}{k}} \norm{f} +c^{\frac{1}{k}} \lambda_{2}^{\frac{1}{k}}\norm{T_{g}T^{\dagger}_{f}f} \quad \big(\text{by Theorem \ref{Theorem 2.8}}\big).
		\end{align*}
		As, $\max{\{\lambda_{1}, \lambda_{2}\}} < \min{\{1, 2^{k-1}\}}$, by using Lemma \ref{P2_L2.5}, $T_{g}T^{\dagger}_{f}$ is invertible, and
		\begin{align*}
			\norm{(T_{g}T^{\dagger}_{f})^{-1}} \leq \frac{1+c^{\frac{1}{k}} \lambda_{2}^{\frac{1}{k}}}{1-c^{\frac{1}{k}} \lambda_{1}^{\frac{1}{k}}}.
		\end{align*}
		Note that  $f=T_{g}T^{\dagger}_{f}(T_{g}T^{\dagger}_{f})^{-1}f = \sum_{n=0}^{\infty} \langle (T_{g}T^{\dagger}_{f})^{-1}f, (T_{f}T^{*}_{f})^{-1}f_{n} \rangle g_{n}$ for every $ f \in \mathcal{H}$.
		Using this, we compute
		\begin{align*}
			\norm{f}^{4} = {\langle f,f \rangle}^{2}  & = \abs{\sum_{n=0}^{\infty} \langle (T_{g}T^{\dagger}_{f})^{-1}f, (T_{f}T^{*}_{f})^{-1}f_{n} \rangle  \langle g_{n}, f \rangle}^{2} \\
			& \leq \sum_{n=0}^{\infty} \abs{\langle (T_{g}T^{\dagger}_{f})^{-1}f, (T_{f}T^{*}_{f})^{-1}f_{n} \rangle }^{2}
			\sum_{n=0}^{\infty} \abs{\langle g_{n},f \rangle }^{2} \\
			& \leq \frac{1}{\alpha}  \norm{(T_{g}T^{\dagger}_{f})^{-1}f}^{2} \sum_{n=0}^{\infty} \abs{\langle g_{n},f \rangle }^{2} \\
			& \leq \frac{1}{\alpha} \left(\frac{1+c^{\frac{1}{k}}\lambda_{2}^{\frac{1}{k}}}{1-c^{\frac{1}{k}}\lambda_{1}^{\frac{1}{k}}}\right)^{2} \norm{f}^{2} \sum_{n=0}^{\infty} \abs{\langle g_{n},f \rangle }^{2}.
		\end{align*}
		Therefore, $\sum_{n=0}^{\infty} \abs{\langle g_{n},f \rangle }^{2} \geq \alpha\left(\frac{1-c^{\frac{1}{k}}\lambda_{1}^{\frac{1}{k}}}{1+c^{\frac{1}{k}}\lambda_{2}^{\frac{1}{k}}}\right)^{2} \norm{f}^{2}$  for all
		$f \in \mathcal{H}$.
		Thus, $\{g_{n}\}_{n=0}^{\infty}$ satisfies the lower frame condition for $\mathcal{H}$. 	Hence, $\{g_{n}\}_{n=0}^{\infty}$  is a frame for $\mathcal{H}$ with  frame bounds
		$\alpha\Big(\frac{1-c^{\frac{1}{k}}\lambda_{1}^{\frac{1}{k}}}{1+c^{\frac{1}{k}}\lambda_{2}^{\frac{1}{k}}}\Big)^{2}$ and $\beta\Big(\frac{1+c^{\frac{1}{k}}\lambda_{1}^{\frac{1}{k}}}{1-c^{\frac{1}{k}}\lambda_{2}^{\frac{1}{k}}}\Big)^{2}$, where $c=\begin{cases}
			1, & \;\text{if} \;  k \geq 1\\
			2^{1-k}, & \;  \text{if} \; k \leq 1 .
		\end{cases}
		$, respectively.
	\end{proof}
Here are a few important remarks on Lemma \ref{L_3.3}, looking at how it connects to earlier work, the role of the exponent $k$, and how it simplifies in special cases.
	
	\begin{remark}\label{rem:PH_lemma_remarks}
		\begin{enumerate}[label=(\roman*)]
			\item \textbf{Reconciliation of the bound $\min\{1, 2^{k-1}\}$ with Definition  \ref{concls009} \eqref{PH_Def}:}
			In Lemma \ref{L_3.3}, the constants $\lambda_1, \lambda_2 \ge 0$ satisfy $\max\{\lambda_1, \lambda_2\} < \min\{1, 2^{k-1}\} = c^{-1}$, where $c = \max\{1, 2^{1-k}\}$ is the constant from Theorem \ref{Theorem 2.8}. Taking $k$-th roots, this is equivalent to requiring
			\[
			\max\{\lambda_1^{\frac{1}{k}}, \lambda_2^{\frac{1}{k}}\} < c^{-\frac{1}{k}} = \min\{1, 2^{1-\frac{1}{k}}\}.
			\]
			Thus, in Definition 1.2(1) (following Retherford \cite{BSPWC}), the bound $\min\{1, 2^{1-1/k}\}$ is stated for the $k$-th roots $\lambda_i^{1/k}$, whereas in Lemma \ref{L_3.3} the bound $\min\{1, 2^{k-1}\}$ is stated directly for the coefficients $\lambda_i$ inside the $k$-th power sum \eqref{P2_L2_1}.
			
\item \textbf{Sharpness of the bound $\min\{1, 2^{k-1}\}$:} The threshold $\min\{1, 2^{k-1}\}$ in Lemma \ref{L_3.3} is sharp for every $0 < k \le 1$. To see this, let $\mathcal{H}=\ell^2(\Z^{+})$ with standard orthonormal basis $\{e_{n}\}_{n \geq 0}$. Consider the sequences
			\begin{align*}
				f_{n}=e_{n} \quad \text{and} \quad g_{n} = e_{n+1}, \quad n \geq 0.
			\end{align*}
			
			The sequence $\{f_{n}\}_{n=0}^{\infty}$ is an orthonormal basis (hence a tight frame) for $\ell^2(\Z^{+})$ . On the other hand, $\{g_{n}\}_{n=0}^{\infty}$ is incomplete in $\ell^2(\Z^{+})$, and thus fails to form a frame.
			Now, for any finite scalar sequence $\{a_n\}_{n=0}^{\infty}$, we have
			\begin{align*}
				& \norm{\sum_{n=0}^{l}a_{n}f_{n}} = \norm{\sum_{n=0}^{l}a_{n}g_{n}} =\left(\sum_{n=0}^{l}\abs{a_n}^{2}\right)^{\frac{1}{2}}; \\
				\intertext{and}
				& \norm{\sum_{n=0}^{l}a_{n}(f_{n}-g_{n})} \leq  \norm{\sum_{n=0}^{l}a_{n}f_{n}} + \norm{\sum_{n=0}^{l}a_{n}g_{n}}= 2 \left(\sum_{n=0}^{l}\abs{a_n}^{2}\right)^{\frac{1}{2}}
			\end{align*}
			For any $0 < k \leq 1$, setting $\lambda_{1} = \lambda_{2} = 2^{k-1}$ satisfies condition \eqref{P2_L2_1} with equality:
			\begin{align*}
				\left(\lambda_{1} \norm{\sum_{n=0}^{l}a_{n}f_{n}}^{k} + \lambda_{2} \norm{\sum_{n=0}^{l}a_{n}g_{n}}^{k}\right)^{\frac{1}{k}}=\left(2. 2^{k-1}\left(\sum_{n=0}^{l}\abs{a_n}^{2}\right)^{\frac{k}{2}}\right)^{\frac{1}{k}} = 2\left(\sum_{n=0}^{l}\abs{a_n}^{2}\right)^{\frac{1}{2}}
			\end{align*}
			but $\{g_{n}\}_{n=0}^\infty$ fails to be a frame. Thus, $\max\{\lambda_1, \lambda_2\} < 2^{k-1}$ is sharp.
			
\item \textbf{Role of the quantifier on $k$ and reduction to \cite[Theorem 2.1]{POATH}:}
			Unlike Definition \ref{concls009} \eqref{PH_Def}  , which is phrased for each $k > 0$, Lemma \ref{L_3.3} requires condition \eqref{P2_L2_1} to hold only \emph{for a given positive real number $k$}. Indeed, applying Theorem \ref{Theorem 2.8} to \eqref{P2_L2_1} yields the linear synthesis-operator estimate
			\[
			\left\|\sum_{n=0}^{l} a_n (f_n - g_n)\right\| \le c^{\frac{1}{k}}\lambda_1^{\frac{1}{k}}\left\|\sum_{n=0}^{l} a_n f_n\right\| + c^{\frac{1}{k}}\lambda_2^{\frac{1}{k}}\left\|\sum_{n=0}^{l} a_n g_n\right\|,
			\]
			which reduces Lemma \ref{L_3.3} to the case $k = 1$ (the Hilding condition) and hence to \cite[Theorem 2.1]{POATH} with constants $\lambda_i = c^{\frac{1}{k}}\lambda_i^{\frac{1}{k}} < 1$ ($i = 1, 2$). Allowing an arbitrary exponent $k > 0$ in \eqref{P2_L2_1} is convenient in applications where $\left\|\sum a_n f_n\right\|^k$ is easier to estimate than $\left\|\sum a_n f_n\right\|$.
		\end{enumerate}
	\end{remark}

	The next result gives a sufficient condition to stably recover the initial state under simultaneous perturbations of the sampling vector and the evolution operator within the framework of Pollard-Hilding-type frame perturbations.
	
	\begin{theorem}\label{P2_Th3.17'}
		Let $g,h \in \mathcal{H}$ and  let $A$ and $B$ be  bounded linear operators acting on  $\mathcal{H}$. Suppose that
		\begin{enumerate}[label=(\roman*)]
			\item 	\label{P2_S4_H31'}any $f \in \mathcal{H}$ is recovered from the samples $\{\langle A^{n}f,g\rangle\}_{n \geq 0}$ in a stable way.
			\item \label{P2_S4_H32'} For a given $k>0$, there exist constants $\lambda_{1}$, $\lambda_{2} \geq 0$ such that 	$\max{\{\lambda_{1}, \lambda_{2}\}} < \min{\{1, 2^{k-1}\}}$ and
			\begin{align}
				\norm{\sum_{n=0}^{l}a_{n}((A^{*})^{n}g-(B^{*})^{n}h)}
				&\leq \left[\lambda_{1} \norm{\sum_{n=0}^{l}a_{n}(A^{*})^{n}g}^{k} + \lambda_{2} \norm{\sum_{n=0}^{l}a_{n}(B^{*})^{n}h}^{k}\right]^{\frac{1}{k}}, \;\; l \in \Z^{+},
				\label{P2_S2_3'}
			\end{align}
			for all finite scalar sequences $\{a_n\}_{n=0}^{\infty}$.
		\end{enumerate}
		Then, any $f \in \mathcal{H}$ can also be recovered from the samples $\{\langle B^{n}f,h\rangle\}_{n \geq 0}$ in a stable way.
	\end{theorem}
	
		\begin{proof}
Let $f_{n}=(A^{*})^{n}g$ and $g_{n}=(B^{*})^{n}h$. Then, by Proposition \ref{P2_P2.3}, the stable recovery of any $f \in \mathcal{H}$ from the dynamical samples $\{\langle A^{n}f,g\rangle\}_{n \geq 0}$ is equivalent to the system $\{(A^{*})^{n}g\}_{n \geq 0}$ forming a frame for $\mathcal{H}$. 	Under the given Pollard--Hilding-type perturbation condition \eqref{P2_S2_3'}, applying Lemma \ref{L_3.3} implies that the perturbed system $\{(B^{*})^{n}h\}_{n \geq 0}$ is also a frame for $\mathcal{H}$. Finally, invoking Proposition \ref{P2_P2.3} for the perturbed system, we conclude that any $f \in \mathcal{H}$ can also be recovered from the perturbed dynamical samples $\{\langle B^{n}f,h\rangle\}_{n \geq 0}$ in a stable way.
	\end{proof}

	The following example demonstrates Theorem \ref{P2_Th3.17'}.
	
	\begin{example}\label{ex:4.3}
		Let $\mathcal{H}=\ell^2(\Z^{+})$ and let $\{e_{n}\}_{n \geq 0}$ be the standard orthonormal basis of $\mathcal{H}$. Consider the left shift operator $A$ given in Example \ref{P2_Exp3.6'} and let $g=e_{0}$. Then,  $\{(A^{*})^{n}g\}_{n \geq 0}=\{e_{n}\}_{n \geq 0}$  is an orthonormal basis of the space $\ell^2(\Z^{+})$, and hence, every $f \in \mathcal{H}$ is stably recovered from the samples $\{\langle A^{n}f,g\rangle\}_{n \geq 0}$.
		
		Now, for any given $k>0$, we have $2^{\frac{k-1}{k}}>0$. By the Archimedian Property, there always exist an $\epsilon \in (0,1)$ depending on $k$ such that such that $\epsilon < 2^{\frac{k-1}{k}} $. For this given $k$, define the perturbed sampling vector
		$$
		h_k = \left(1 - \frac{\epsilon}{2}\right)e_{0},
		$$
		and let $B_k:  \ell^2(\Z^{+}) \rightarrow \ell^2(\Z^{+})$ be the bounded linear operator whose Hilbert-adjoint $B_k^*$ depends on $k$ through $\epsilon$ and is given by
		$$
		B_k^{*}(x_0, x_1, x_2, \ldots) = (0, \beta_0 x_0, \beta_1 x_1, \beta_2 x_2, \dots), \quad \text{where } \beta_i = 1 - \frac{\epsilon}{2^{i+2}} \quad (i \geq 0).
		$$
		Then,  $\{(B_k^{*})^{n} h_k\}_{n \geq 0}=\{P_{n}e_{n}\}_{n \geq 0}$ and  $\{(A^{*})^{n}g-(B^{*})^{n}h_k\}_{n \geq 0}=\{(1-P_{n})e_{n}\}_{n \geq 0}$, where $P_{0}= 1 - \frac{\epsilon}{2}$ and $P_{n}=  \left(1 - \frac{\epsilon}{2}\right) \prod_{i=0}^{n-1} \beta_{i}$ or $n \geq 1$. Also, for any finite scalar sequence $\{a_n\}_{n=0}^{\infty}$ and $ l \in \Z^{+}$, we have
		\begin{align}
			& \norm{\sum_{n=0}^{l}a_{n}(A^{*})^{n}g}^{k} =\norm{\sum_{n=0}^{l}a_{n}e_{n}}^{k} =\left(\sum_{n=0}^{l}\abs{a_n}^{2}\right)^{\frac{k}{2}}; \label{lk441'}\\
			& \norm{\sum_{n=0}^{l}a_{n}((A^{*})^{n}g-(B^{*})^{n}h_k)}^{k} =\norm{\sum_{n=0}^{l}a_{n}(1- P_{n})e_{n}}^{k} =\left(\sum_{n=0}^{l}\abs{a_n(1-P_{n})}^{2}\right)^{\frac{k}{2}}; \label{lk442'}
			\intertext{and}
			& \norm{\sum_{n=0}^{l}a_{n}(B_k^{*})^{n} h_k}^{k} =\norm{\sum_{n=0}^{l}a_{n}(P_{n}e_{n})}^{k}
			=\left(\sum_{n=0}^{l}\abs{a_n P_{n}}^{2}\right)^{\frac{k}{2}}. \label{lk443'}
		\end{align}
		
		Since $1 - \frac{\epsilon}{2} < \prod_{i=0}^{n-1} \beta_{i} \leq 1$, we have  $1-\epsilon < P_{n} \leq 1- \frac{\epsilon}{2}$ which implies $\frac{\epsilon}{2} \leq 1-P_{n}< \epsilon$ for all $n \geq 0$.
		Thus,
		\begin{align*}
			\left(\sum_{n=0}^{l}\abs{a_n(1-P_{n})}^{2}\right)^{\frac{k}{2}}  < \epsilon^{k} \left(\sum_{n=0}^{l}\abs{a_n}^{2}\right)^{\frac{k}{2}}.
		\end{align*}
		Combining this with \eqref{lk441'}, \eqref{lk442'}, and \eqref{lk443'}, we obtain
		\begin{align*}
			\norm{\sum_{n=0}^{l}a_{n}((A^{*})^{n}g-(B^{*})^{n} h_k)}^{k} 	 & =\left(\sum_{n=0}^{l}\abs{a_n(1-P_{n})}^{2}\right)^{\frac{k}{2}} \\
			& \leq \epsilon^{k}\left(\sum_{n=0}^{l}\abs{a_n}^{2}\right)^{\frac{k}{2}} + \epsilon^{k}  \left(\sum_{n=0}^{l}\abs{a_n P_{n}}^{2}\right)^{\frac{k}{2}}\\
			& = \lambda_{1} \norm{\sum_{n=0}^{l}a_{n}(A^{*})^{n}g}^{k} + \lambda_{2} \norm{\sum_{n=0}^{l}a_{n}(B_k^{*})^{n}h_k}^{k},
		\end{align*}
		where  $\lambda_{1}= \lambda_{2}=\epsilon^{k}$. Since $\epsilon < 2^{\frac{k-1}{k}}$, we have $\max\{\lambda_{1}, \lambda_{2}\}= \epsilon^{k} < \min{\{1, 2^{k-1}\}}$. Hence, for the given $k > 0$, all conditions of Theorem \ref{P2_Th3.17'} are satisfied, and any $f \in \mathcal{H}$ can be stably recovered from the samples $\{\langle B_k^{n}f,h_k\rangle\}_{n \geq 0}$.
		
	\end{example}
	
	As a direct result, the following corollary shows what happens when only  the sampling vector or  the evolution operator.
	
	\begin{corollary}\label{cor_PH}
		Let $A, B \in \mathcal{B}(\mathcal{H})$ and $g, h \in \mathcal{H}$ be such that any $f \in \mathcal{H}$ is recovered from the samples $\{\langle A^{n}f,g\rangle\}_{n \geq 0}$ in a stable way.  Suppose that for a given positive real number $k$, there exist constants $\lambda_{1}$, $\lambda_{2} \geq 0$ such that $\max{\{\lambda_{1}, \lambda_{2}\}} < \min{\{1, 2^{k-1}\}}$.
		\begin{enumerate}[label=(\roman*)]
			\item \label{cor_PH_1} If for all $l \in \Z^{+}$  and all finite scalar sequences $\{a_n\}_{n=0}^{\infty}$,
			\begin{align*}
				\norm{\sum_{n=0}^{l}a_{n}(A^{*})^{n}(g-h)} \leq \left[\lambda_{1} \norm{\sum_{n=0}^{l}a_{n}(A^{*})^{n}g}^{k} + \lambda_{2} \norm{\sum_{n=0}^{l}a_{n}(A^{*})^{n}h}^{k}\right]^{\frac{1}{k}},
			\end{align*}
			then any $f \in \mathcal{H}$ can also be recovered from the samples $\{\langle A^{n}f,h\rangle\}_{n \geq 0}$ in a stable way.
			
			\item \label{cor_PH_2} If for all $l \in \Z^{+}$  and all finite scalar sequences $\{a_n\}_{n=0}^{\infty}$,
			\begin{align*}
				\norm{\sum_{n=0}^{l}a_{n}((A^{*})^{n}-(B^{*})^{n})g} \leq \left[\lambda_{1} \norm{\sum_{n=0}^{l}a_{n}(A^{*})^{n}g}^{k} + \lambda_{2} \norm{\sum_{n=0}^{l}a_{n}(B^{*})^{n}g}^{k}\right]^{\frac{1}{k}},
			\end{align*}
			then any $f \in \mathcal{H}$ can also be recovered from the samples $\{\langle B^{n}f,g\rangle\}_{n \geq 0}$ in a stable way.
		\end{enumerate}
	\end{corollary}
	
	\begin{proof}
		Parts \ref{cor_PH_1} and \ref{cor_PH_2} follow immediately from Theorem \ref{P2_Th3.17'} by setting $B = A$ and $h = g$, respectively.
	\end{proof}

	In real-world problems, measurement errors (e.g., sensor drift) and system changes often happen separately. To handle these issues, we next present a split-sum version of the Theorem \ref{P2_Th3.17'}.

		\begin{theorem}\label{P2_Th3.17}
		Let $g,h \in \mathcal{H}$ and  let $A$ and $B$ be  bounded linear operators acting on  $\mathcal{H}$. Suppose that
		\begin{enumerate}[label=(\roman*)]
			\item 	\label{P2_S4_H31}any $f \in \mathcal{H}$ is recovered from the samples $\{\langle A^{n}f,g\rangle\}_{n \geq 0}$ in a stable way.
			\item \label{P2_S4_H32} For a given $k>0$, there exist constants $\lambda_{1}$, $\lambda_{2} \geq 0$ such that 	$\max{\{\lambda_{1}, \lambda_{2}\}} < \min{\{1, 2^{k-1}\}}$ and
			\begin{align}
				\notag
				& \norm{\sum_{n=0}^{l}a_{n}(A^{*})^{n}(g-h)} + \norm{\sum_{n =0}^{l}a_{n}((A^{*})^{n}-(B^{*})^{n})h} \\
				&\leq \left[\lambda_{1} \norm{\sum_{n=0}^{l}a_{n}(A^{*})^{n}g}^{k} + \lambda_{2} \norm{\sum_{n=0}^{l}a_{n}(B^{*})^{n}h}^{k}\right]^{\frac{1}{k}}, \;\; l \in \Z^{+},
				\label{P2_S2_3}
			\end{align}
			for all finite scalar sequences $\{a_n\}_{n=0}^{\infty}$.
		\end{enumerate}
		Then, any $f \in \mathcal{H}$ can also be recovered from the samples $\{\langle B^{n}f,h\rangle\}_{n \geq 0}$ in a stable way.
	\end{theorem}
	
		\begin{proof}
		By Proposition \ref{P2_P2.3}, the stable recovery of any $f \in \mathcal{H}$ from the dynamical samples $\{\langle A^{n}f,g\rangle\}_{n \geq 0}$ is equivalent to the system $\{(A^{*})^{n}g\}_{n \geq 0}$ forming a frame for $\mathcal{H}$.
		
		To show that the perturbed system $\{(B^{*})^{n}h\}_{n \geq 0}$ satisfies the frame conditions under the given hypotheses, we first apply the triangle inequality to the orbit difference:
		\begin{align*}
			\norm{\sum_{n=0}^{l}a_{n}((A^{*})^{n}g-(B^{*})^{n}h)} & = \norm{\sum_{n=0}^{l}a_{n}((A^{*})^{n}g-(A^{*})^{n}h+(A^{*})^{n}h-(B^{*})^{n}h)} \notag\\
			& \leq 	\norm{\sum_{n=0}^{l}a_{n}(A^{*})^{n}(g-h)} + \norm{\sum_{n =0}^{l}a_{n}((A^{*})^{n}-(B^{*})^{n})h} \notag \\
			& \leq \left[\lambda_{1} \norm{\sum_{n=0}^{l}a_{n}(A^{*})^{n}g}^{k} + \lambda_{2} \norm{\sum_{n=0}^{l}a_{n}(B^{*})^{n}h}^{k}\right]^{\frac{1}{k}}.
		\end{align*}
		for all finite scalar sequences $\{a_n\}_{n=0}^{\infty}$ and $l \in \Z^{+}$, where $\max{\{\lambda_{1}, \lambda_{2}\}} < \min{\{1, 2^{k-1}\}}$.
		Using \eqref{P2_S2_3} in the above inequality gives the general frame perturbation framework of Lemma \ref{L_3.3}, it follows that the perturbed system $\{(B^{*})^{n}h\}_{n \geq 0}$ is also a frame for $\mathcal{H}$. Finally, invoking Proposition \ref{P2_P2.3} for the perturbed system, we conclude that any $f \in \mathcal{H}$ can also be recovered from the perturbed dynamical samples $\{\langle B^{n}f,h\rangle\}_{n \geq 0}$ in a stable way.
	\end{proof}

	The following remark compares Theorem \ref{P2_Th3.17'} and Theorem \ref {P2_Th3.17}, highlighting the difference between using a single combined condition versus a split version.
	
		\begin{remark}\label{rem:PH-comparison}
		Similar to the Nagy-type framework, Theorem \ref{P2_Th3.17'} and Theorem \ref {P2_Th3.17} offer two distinct viewpoints on simultaneous Pollard-Hilding-Type frame perturbations: Theorem \ref{P2_Th3.17'} provides a concise condition on the total combined orbit difference $((A^{*})^{n}g-(B^{*})^{n}h)$, capturing the overall stability envelope of the dynamical sampling process in a single inequality, whereas Theorem \ref{P2_Th3.17} separates the perturbation into the individual contributions of the sampling vector error $g - h$ and the evolution operator error $((A^{*})^{n}-(B^{*})^{n})$. This split formulation is particularly useful when experimental errors can be measured or bounded separately for the sensing device and the propagation medium.
	\end{remark}

	\section{Conclusion}
	A sequence  $\{\psi_{i}\}_{i=1}^{\infty}$  in a Hilbert space $\mathcal{H}$ is said to be \emph{complete} if $\overline{\text{span}}\{\psi_{i}\}_{i=1}^{\infty} = \mathcal{H}$. Harry Pollard, in  \cite{Pollard}, proved a Paley-Wiener-type stability result for complete sequences in Hilbert spaces. Further, he also gives a criterion  for incomplete sequences in Hilbert spaces in terms of a Paley-Wiener-type condition. After four years, Hilding \cite{NCTPWT} presented a Paley-Wiener-type stability result for completeness of a sequence  in terms of a linear transformation on Hilbert space.
	Near the middle of the 20th century, Retherford \cite{BSPWC} generalized stability results in the sense of Paley-Wiener for basic sequences and completeness of sequences in complete linear metric spaces and normed spaces and introduced Nagy-type and  Pollard--Hilding-type stability conditions  [Definition \ref{concls009}].
	Casazza and Kalton, in \cite{Casz},  extended the Paley-Wiener-type stability results to linear operators on Banach spaces.  Frames in separable Hilbert spaces are complete sequences which satisfy an operator inequality, also known as the frame inequality, that is a powerful tool in both pure and engineering science.   Paley-Wiener type stability conditions for frames and Riesz bases were reviewed in \cite{C.Heil, Young}. Most recently, the authors of \cite{JyotiandR} studied Nagy-type and  Pollard--Hilding-type perturbations of Hilbert-Schmidt frames.

	On the other hand, Aldroubi et al. \cite{DSTS, DSSIO, ICNO} studied conditions
	on $A$, $\mathcal{G}$ and a function $L : \mathcal{G} \rightarrow \N \cup \{\infty\}$, that allow the stable recovery of the initial state of the homogeneous discrete dynamical system of the form $f_n = A f_{n-1}=A^{n}f, f_0 = f$, from the set of samples
	$	\{\langle A^n f, g \rangle  : g \in \mathcal{G}, 0 \leq  n < L(g)\}$, where $A$ is the evolution operator, $f_0$ is the initial state, and $\mathcal{G} \subset \mathcal{H}$, is a finite or countably infinite set.

	In this work, we give applications of Nagy-type and  Pollard--Hilding-type stability conditions to homogeneous dynamical sampling systems for stable recovery of the initial state in terms of frames. More precisely, we establish sufficient conditions for the stable recovery of an initial state from perturbed dynamical samples, which are derived by altering the sampling vector, the evolution operator, or both concurrently within the context of Nagy-type and Pollard–Hilding-type perturbations of frames. A further interesting direction would be applications of perturbation of weaving frames \cite{DeepI, deep2} in the stable recovery of source terms and initial states in  dynamical sampling systems.
	

	\textbf{Ruchi}, Department of Mathematics,
	University of Delhi, Delhi-110007, India.\\
	Email: rgarg@maths.du.ac.in\\
	
	\textbf{Lalit Kumar Vashisht}, Department of Mathematics,
	University of Delhi, Delhi-110007, India.\\
	Email: lalitkvashisht@gmail.com
\end{document}